\documentclass[a4paper,11pt,UKenglish]{article}
\usepackage[top=3.2cm, bottom=3cm, left=3.45cm, right=3.45cm]{geometry}
\usepackage[T1]{fontenc}
\usepackage[utf8]{inputenc}
\usepackage{latexsym}
\usepackage{amsmath,amsthm,amssymb,amsfonts,mathtools}
\usepackage[]{babel}
\usepackage{tikz}
\usepackage{multirow}
\usepackage{booktabs}
\usepackage{multicol}
\usepackage[shortlabels]{enumitem}
\usepackage{scalerel}
\usepackage{hyperref}

\newcommand{\LL}{\mathbb{L}}
\newcommand{\BB}{\mathbb{B}}
\newcommand{\Set}{E}
\newcommand{\Lin}{L}
\newcommand{\Sym}{\mathcal{S}}
\newcommand{\End}{\mathrm{End}}
\newcommand{\Cyc}{\mathcal{C}}
\newcommand{\Der}{\mathrm{Der}}

\newcommand{\rtree}{\mathcal{A}}
\newcommand{\tree}{\mathfrak{a}}
\newcommand{\btree}{\mathcal{B}}

\newcommand{\Par}{\mathrm{Par}}
\newcommand{\Bal}{\mathrm{Bal}}
\newcommand{\Cay}{\mathrm{Cay}}

\newcommand{\id}{\mathrm{id}}
\newcommand{\Img}{\mathrm{Im}}
\newcommand{\ddy}[1]{\displaystyle{\frac{\partial{#1}}{\partial Y}}}
\newcommand{\ddx}[1]{\displaystyle{\frac{\partial{#1}}{\partial X}}}
\newcommand{\pointX}{^{\bullet_X}}
\newcommand{\pointY}{^{\bullet_Y}}
\newcommand{\stirl}[3]{\genfrac{\{}{\}}{0pt}{}{{#1}}{{#2}}_{{#3}}}
\newcommand{\sdiff}[3]{\genfrac{\{}{\}}{}{}{{#1}}{{#2}}_{{#3}}}

\newcommand{\Har}[1]{H_{#1}}
\newcommand{\ff}[2]{(#1)_{#2}}
\newcommand{\setstirling}[2]{%
  \genfrac{\{}{\}}{0pt}{}{{#1}}{{#2}}}
\newcommand{\fub}[1]{\mathfrak{f}(#1)}
\newcommand{\rfub}[2]{\mathfrak{f}_{#2}(#1)}
\newcommand{\sfub}[2]{%
  \overset{\raisebox{-0.2ex}{\rule{0.41em}{0.05em}}}%
  {\mathfrak{f}}_{#2}(#1)}
\newcommand{\CayDer}{\Cay_{\Der}}
\newtheorem{theorem}{Theorem}[section]
\newtheorem{proposition}[theorem]{Proposition}
\newtheorem{lemma}[theorem]{Lemma}
\newtheorem{corollary}[theorem]{Corollary}

\newtheorem*{openproblem*}{Open Problem}

\theoremstyle{definition}
\newtheorem{definition}[theorem]{Definition}

\newtheorem*{remark*}{Remark}

\newtheorem*{example*}{Example}

\title{Counting fixed-point-free Cayley permutations}
\date{13 March 2026}
\author{Giulio Cerbai \;and\; Anders Claesson \vspace{1.5ex}\\
Department of Mathematics\vspace{0.15ex}\\ University of Iceland}

\begin{document}
\maketitle
\thispagestyle{empty}

\begin{abstract}
Two-sort species yield differential equations for functional digraphs
of Cayley permutations.
From these we obtain an explicit formula for fixed-point-free Cayley
permutations and prove that their proportion tends to $1/e$, as
for permutations and endofunctions. Our approach also yields counting
formulas when the functional digraph is a tree, forest, or connected.
\end{abstract}

\section{Introduction}

The \emph{hat-check problem} lies at the intersection of combinatorics
and probability theory. Pierre Rémond de Montmort~\cite{Mon} first
described it in 1708 in his ``Essay d'analyse sur les jeux de hazard''.
The problem asks for the probability that a uniformly random permutation
is a \emph{derangement}, that is, fixed-point-free. Montmort and Bernoulli
later provided solutions featuring a striking occurrence of Napier's
constant in combinatorics: The number of derangements of size $n$ is
$n!\bigl(1/0! - 1/1! + 1/2! - \cdots + (-1)^n/n!\bigr)$
and hence, as $n\to\infty$, the probability converges to $1/e$.

A lesser-known yet easier-to-prove result is that \emph{endofunctions}
have the same asymptotic behaviour. Here endofunctions are maps from
$[n]=\{1,2,\ldots,n\}$ to $[n]$. Among all $n^n$ endofunctions,
$(n-1)^n$ have no fixed points. As $n$ approaches infinity, the quotient
converges again to $1/e$. Indeed, $(n-1)^n/n^n = (1-1/n)^{n}\to 1/e$.

Cayley permutations lie between permutations and endofunctions and
present a more challenging version of this problem. They are
endofunctions whose image contains every integer from one to its
maximum, or equivalently, they are the lexicographically smallest
representatives of positive integer words modulo order isomorphism.
Cayley permutations play a key role in recent extensions of
permutation pattern theory~\cite{BBO,CeCay,CCtop,CCcaypol,CCcaypol2}.

The difficulty of \emph{Cayley-derangements}---Cayley permutations
with no fixed points---lies in their biased nature. In permutations
and endofunctions, fixed points are uniformly distributed. In Cayley
permutations, smaller values appear more often and are thus more
likely to be fixed. Furthermore, the combinatorial description is
less transparent: a permutation decomposes as its fixed points
together with a derangement of the rest, leading directly to
enumeration by subfactorials, but no such simple decomposition
exists for Cayley permutations.


This paper provides a counting formula for Cayley-derangements
(Theorem~\ref{thm_eq_cayder}). Namely, the number of Cayley-derangements
on $[n]$ is
\[
  \sum_{r=0}^n\frac{|\Der[r]|}{r!}\sum_{i=0}^ni!\sdiff{n}{i}{r}
\]
in which $|\Der[r]|$ is the $r$th subfactorial and
$\sdiff{n}{i}{r} = \stirl{n}{i}{r}-\stirl{n}{i}{r+1}$ is a difference
of $r$-Stirling numbers. Using this explicit formula,
we prove (Theorem~\ref{main-theorem}) that the probability
that a uniformly random Cayley permutation is fixed-point-free
is $1/e$, as for permutations and endofunctions.

We use two-sort species to study functional digraphs of
Cayley-derangements, deriving differential equations for the broader
class of \emph{$R$-recurrent} Cayley permutations. These include
Cayley permutations whose functional digraph is a tree, a forest,
or connected. The appearance of species and functional digraphs
here echoes Joyal's~\cite{J} proof of Cayley's formula.
We outline the paper below.
Section~\ref{section_species} reviews combinatorial species and
functional digraphs. Section~\ref{section_cayperms} introduces Cayley
permutations. Section~\ref{section_Rrecdig} constructs $R$-recurrent
functional digraphs, and Section~\ref{section_diffeq} derives the
differential equations yielding our counting formula.
Section~\ref{section_asymp} proves that the proportion of
Cayley-derangements tends to $1/e$.
Section~\ref{section_Joyal} extends Joyal's bijection to the
two-sort setting, while Section~\ref{section_unisort} provides
complementary unisort equations. Section~\ref{section_stats} studies statistics
such as internal nodes and connected components.
Section~\ref{section_generalizations} discusses generalizations.


\section{Combinatorial species and functional digraphs}\label{section_species}

A \emph{$\BB$-species} (or simply \emph{species}) is a functor
$F:\BB\to\BB$, where $\BB$ is the category of finite sets with
bijections as morphisms.  A species defines both a class of labeled
combinatorial structures---$F[U]$---and how relabeling of the
underlying set affects these structures---the transport of structure
$F[\sigma]$. Since the transport of structure is a bijection, the
cardinality of $F[U]$ depends solely on the cardinality of $U$, not on
the nature of the elements of $U$.  The \emph{generating series} of a
species $F$ is the formal power series
$$
F(x)=\sum_{n\ge 0}|F[n]|\frac{x^n}{n!},
$$
where $F[n]$ is short for $F[[n]]$ and $[n]=\{1,2,\dots,n\}$ (in
particular, $[0]=\emptyset$).

Throughout this paper we use the following species, whose formal
definitions can be found in Joyal's seminal paper~\cite{J} or in
the book by Bergeron, Labelle and Leroux~\cite{BLL}. These are the main references
for combinatorial species theory. Shorter introductions
can be found in our papers~\cite{CCcaypol2,CCEG,Cl}.
\begin{description}[labelindent=2em]
\item[$1=\Set_0$] is the species characteristic of the empty set.
\item[$X=\Set_1$] is the species characteristic of singletons.
\item[$\Set$] is the species of sets.
\item[$\Lin$] is the species of linear orders.
\item[$\Sym$] is the species of permutations.
\item[$\Cyc$] is the species of cyclic permutations (cycles).
\end{description}
Their generating series are
$1(x)=1$, $X(x)=x$,
\[
  \Set(x)=e^x,\;
  \Lin(x)=\Sym(x)=1/(1-x)\,\text{ and }\,
  \Cyc(x)=-\log\bigl(1-x\bigr).
\]
Two species $F$ and $G$ are \emph{(combinatorially) equal}, written
$F=G$, if there is a natural isomorphism between them.  Equality implies
$F(x)=G(x)$, but not conversely: $\Lin(x)=\Sym(x)$ while $\Lin\neq \Sym$.


The species $G$ is a \emph{subspecies} of $F$, written $G\subseteq F$,
if for each pair of finite sets $U,V$ and bijection $\sigma:U\to V$ we have
$G[U]\subseteq F[U]$ and $G[\sigma]=F[\sigma]|_{G[U]}$.
We denote by $F_+$ and $F_n$ the subspecies of $F$ consisting of
$F$-structures on nonempty sets and on sets with cardinality $n$,
respectively.

%
%

Species can be viewed as a categorification of exponential generating
series: each operation on formal power series---sum, product,
composition, differentiation---has a canonical lift to species that
recovers the original operation upon taking generating series. We
now describe these lifts, defining only the set of structures; the
transport of structure is implicit in each case, and a formal
treatment can be found in the book by Bergeron et al.~\cite{BLL}.

Given two species $F$ and $G$, an \emph{$(F+G)$-structure} is either an
$F$-structure or a $G$-structure:
\[
  (F+G)[U]=F[U] \sqcup G[U],
\]
where $\sqcup$ denotes disjoint union. An \emph{$(F\cdot G)$-structure}
on $U$ is a pair $(s,t)$ such that $s$ is an $F$-structure on a subset
$U_1\subseteq U$ and $t$ is a $G$-structure on $U_2=U\setminus U_1$:
\[
(F \cdot G)[U] = \bigsqcup_{(U_1,U_2)} F[U_1] \times G[U_2],
\]
where $U=U_1 \sqcup U_2$ and ``$\times$'' denotes cartesian product.

An \emph{$(F\circ G)$-structure}, where
$G[\emptyset] = \emptyset$, is a generalized partition where each block
carries a $G$-structure and the set of blocks is structured by $F$:
\[
(F \circ G)[U] = \bigsqcup_{\beta\in\Par[U]}
F[\beta]\times \prod_{B\in\beta}G[B],
\]
where $\Par[U]$ denotes the set of partitions of $U$.  For instance, the
species of \emph{ballots} (ordered set partitions) is defined by
$\Bal = \Lin\circ\Set_+$ and hence $\Bal(x)=1/(2-e^x)$; here
$|\Bal[n]|$ is the $n$th Fubini number (A000670~\cite{Sl}).

The \emph{derivative} $F'$, defined by $F'[U] = F[U \sqcup \{\star\}]$,
adds an element to the underlying set.
\emph{Pointing} $F^{\bullet}$, defined by $F^{\bullet}[U] = F[U]\times U$,
distinguishes an element of the set. The two are related by $F^{\bullet} = X\cdot F'$.

The species of \emph{endofunctions} is defined by
$$
\End[U]=\{f:U\to U\}
\quad\text{and}\quad
\End[\sigma](f)=\sigma\circ f\circ\sigma^{-1}
$$
for any finite sets $U$ and $V$ and any bijection $\sigma:U\to V$.
Throughout this paper we identify an endofunction $f\in\End[U]$
with its \emph{functional digraph}, defined as the directed graph with
vertex set~$U$ and directed edges $(u,f(u))$, or $u\mapsto f(u)$,
$u\in U$.
Alternatively, the species of functional digraphs is the subspecies
of directed graphs where each node has outdegree one.
The transport of structures is inherited from the
transport on directed graphs: a directed graph on~$U$ with edges
$\{(x,y)\}$ is mapped by the transport of structure
along~$\sigma$ to the directed graph on~$V$ with edges
$\{(\sigma(x),\sigma(y)\}$.
Identifying an endofunction with its functional digraph is a natural
isomorphism, and thus the species of endofunctions and functional
digraphs are combinatorially equal. An instance of a functional
digraph can be found in Figure~\ref{figure_endofun}.

\begin{figure}
\centering
\begin{tikzpicture}[
dot/.style = {draw,circle,inner sep=1.25pt, node contents={}, label=#1},
every loop/.style={min distance=7mm,in=50,out=130,looseness=10},
thick,
baseline=0pt,
scale=1]
\begin{scope}[font = {\small}]
      \node[fill=black] (3) at (0.5,2.75) [dot=above:$3$];
      \node[fill=black] (7) at (0,2) [dot=left:$7$];
      \node[fill=black] (5) at (1,2) [dot=right:$5$];
      \node[fill=black] (4) at (0,1) [dot=left:$4$];
      \node (10) at (-0.33,0) [dot=below:$10$];
      \node (14) at (0.33,0) [dot=below:$14$];
      \node (11) at (0.66,1) [dot=below:$11$];
      \node (13) at (1.33,1) [dot=below:$13$];
      \path (7) edge[-stealth, bend left] (3);
      \path (3) edge[-stealth, bend left] (5);
      \path (5) edge[-stealth, bend left] (7);
      \path (4) edge[-stealth] (7);
      \path (10) edge[-stealth] (4);
      \path (14) edge[-stealth] (4);
      \path (11) edge[-stealth] (5);
      \path (13) edge[-stealth] (5);
      \node[fill=black] (6) at (3.1,2) [dot=left:$6$];
      \node[fill=black] (9) at (3.1,1) [dot=left:$9$];
      \node (1) at (2.76,0) [dot=below:$1$];
      \node (12) at (3.43,0) [dot=below:$12$];
      \path (6) edge[loop above, -stealth] (6);
      \path (9) edge[-stealth] (6);
      \path (1) edge[-stealth] (9);
      \path (12) edge[-stealth] (9);
      \node[fill=black] (2) at (5,2) [dot=left:$2$];
      \node[fill=black] (8) at (5.66,2) [dot=right:$8$];
      \node (15) at (5.66,1) [dot=left:$15$];
      \path (2) edge[-stealth, bend left=60] (8);
      \path (8) edge[-stealth, bend left=60] (2);
      \path (15) edge[-stealth] (8);
\end{scope}
\end{tikzpicture}
\caption{The functional digraph of $f=985776326459548\in\End[15]$,
  where the $i$th letter of~$f$ is $f(i)$.
  Internal nodes are black while leaves are white.}
\label{figure_endofun}
\end{figure}
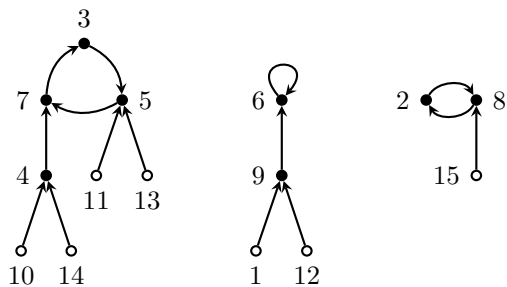

Let $u\in U$ be a node in the functional digraph of~$f$.
We say that $u$ is \emph{internal} if it has positive indegree, that is,
if $u\in\Img(f)$. Otherwise, $u$ is a \emph{leaf}.
A \emph{fixed point} of~$f$ is an element $u\in U$ such that $f(u)=u$,
which corresponds to a \emph{loop} in the functional digraph.

The species of \emph{permutations} $\Sym\subseteq\End$ consists of
bijective endofunctions, and \emph{cycles} $\Cyc\subseteq\Sym$ are
permutations whose functional digraph is a single cycle. Every
permutation decomposes uniquely into disjoint cycles, giving
$\Sym = \Set(\Cyc)$ and
\[
  \Cyc(x)=\log\Sym(x)=\log\frac{1}{1-x}.
\]
A \emph{derangement} is a fixed-point-free permutation; the species
$\Der\subseteq\Sym$ satisfies $\Sym = \Set\cdot\Der$ since a permutation
is a set of fixed points together with a derangement of the rest.
Solving for $\Der(x)$ yields
\[
\Der(x)= \frac{e^{-x}}{1-x}
\quad\text{and}\quad
|\Der[n]|=n!\sum_{i=0}^n\frac{(-1)^i}{i!}.
\]
The species $\rtree$ of rooted trees plays a key role in this paper.
A \emph{tree} is a connected acyclic simple graph
and we let $\tree$ denote the species of trees.
A \emph{rooted tree} is a tree with a distinguished node, called its
\emph{root}; equivalently, a pointed tree,
$\rtree=\tree^{\bullet}$. A rooted tree consists of
a root (an $X$-structure) to which is appended a set of rooted trees (an $\Set(\rtree)$-structure).
That is,
\[
  \rtree \,=\, X\cdot \Set(\rtree).
\]
A slight tweak of the above definition allows us to see rooted trees as
a subspecies of functional digraphs. Note that there is exactly one
cycle in every connected component of a functional digraph.  We identify
rooted trees with connected functional digraphs whose only cycle is a
loop (a fixed point); this loop plays the role of root, and the edges lose
their orientation.  Conversely, by putting a loop around the root of a
tree and orienting the edges from the leaves towards the root we obtain
a connected functional digraph whose only cycle is a loop. A rooted tree
and its digraph counterpart are depicted in Figure~\ref{figure_rtree}.

\begin{figure}
\centering
\begin{tikzpicture}
[every loop/.style={min distance=7mm,in=50,out=130,looseness=10},
every node/.style={fill=black,draw,circle,inner sep=1.5pt,outer sep=1pt},
thick, baseline=0pt, scale=0.8]
\begin{scope}[font = {\small}]
\node[label={[left]3~~}](3) at (1.5,3){};
\node[label={[left]6}](6) at (.5,2){};
\node[label={[left]8}](8) at (1.5,2){};
\node[label={[right]9}](9) at (2.5,2){};
\node[label={[left]1}](1) at (0,1){};
\node[label={[right]5}](5) at (1,1){};
\node[label={[left]2}](2) at (2,1){};
\node[label={[right]7}](7) at (3,1){};
\node[label={[left]4}](4) at (0,0){};
\path (3) edge[loop above, -stealth] (3);
\path (4) edge[-stealth] (1);
\path (1) edge[-stealth] (6);
\path (5) edge[-stealth] (6);
\path (6) edge[-stealth] (3);
\path (8) edge[-stealth] (3);
\path (9) edge[-stealth] (3);
\path (2) edge[-stealth] (9);
\path (7) edge[-stealth] (9);
\end{scope}
\end{tikzpicture}
\qquad\qquad
\begin{tikzpicture}
[every node/.style={fill=black,draw,circle,inner sep=1.25pt},
thick, baseline=0pt,scale=0.8]
\begin{scope}[font = {\small}]
\node[label={[left]3~~}, outer sep=2pt](3) at (1.5,3){};
\draw (1.5,3) circle (4pt);
\node[label={[left]6}](6) at (.5,2){};
\node[label={[left]8}](8) at (1.5,2){};
\node[label={[right]9}](9) at (2.5,2){};
\node[label={[left]1}](1) at (0,1){};
\node[label={[right]5}](5) at (1,1){};
\node[label={[left]2}](2) at (2,1){};
\node[label={[right]7}](7) at (3,1){};
\node[label={[left]4}](4) at (0,0){};
\draw[-] (4) -- (1) -- (6) -- (3) -- (8);
\draw[-] (6) -- (5);
\draw[-] (3) -- (9) -- (2);
\draw[-] (9) -- (7);
\end{scope}
\end{tikzpicture}
\caption{The functional digraph of $f=693163933\in \End[9]$, on the left,
and the corresponding rooted tree, on the right.
}\label{figure_rtree}
\end{figure}
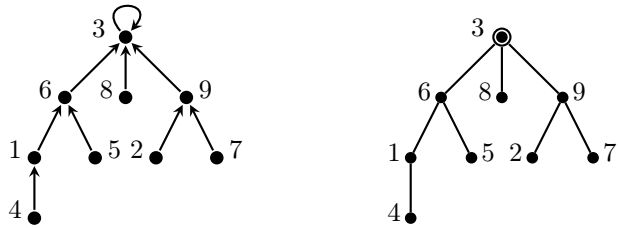

A formula first discovered by Borchardt~\cite{Bor} and later attributed
to Cayley~\cite{Cay} shows that there are $n^{n-2}$ labeled trees
on $[n]$. Joyal~\cite{J} exploited the interplay between rooted trees
and endofunctions to prove Cayley's formula (see also Bergeron et al.~\cite[Section~2.1]{BLL}).
We recall the core idea, which we extend to two-sort species in
Section~\ref{section_Joyal}.
First, observe that an
endofunction is a permutation of rooted trees:
\begin{equation}\label{eq_end=permrtree}
\End=\Sym(\rtree).
\end{equation}
To see this, let us distinguish two types of points in its functional
digraph:
\begin{itemize}
\item \emph{Recurrent points}: $u$ such that $f^k(u)=u$ for some
  $k\ge 1$; they are the points on cycles. Together they form
  a set of cycles, that is, a permutation.
\item \emph{Nonrecurrent points}: $u$ such that $f^k(u)\neq u$ for every
  $k\ge 1$; they are arranged in (rooted) trees each of which hangs from
  a recurrent point, which can be seen as its root.
\end{itemize}
The underlying set $U$ is thus partitioned into a set of rooted trees,
structured by a permutation of recurrent points. For example, the
permutation of the recurrent points of the endofunction in
Figure~\ref{figure_endofun} is
\[
\begin{array}{llllll}
  2\mapsto 8;&
  3\mapsto 5;&
  5\mapsto 7;&
  6\mapsto 6;&
  7\mapsto 3;&
  8\mapsto 2.
\end{array}
\]

Joyal proved that doubly-rooted trees and nonempty linear
orders of rooted trees are combinatorially equal. Indeed, let $t$ be a
structure of $\rtree^{\bullet}=\tree^{\bullet\bullet}$; that is,
a tree where two nodes (not necessarily distinct) are distinguished.
Define the \emph{spine} of~$t$ as the directed path from the first
distinguished node, the \emph{tail}, to the second distinguished node, the
\emph{head}. Joyal used the suggestive name ``vertebrates''
for the species of doubly-rooted trees.
The spine determines a nonempty linear order, with a (rooted) tree
hanging from each of its nodes. The resulting structure is a
nonempty linear order of rooted trees:
\[\rtree^{\bullet}=\Lin_+(\rtree).
\]
Two species $F$ and $G$ are \emph{equipotent}, written $F\equiv G$, if $F(x)=G(x)$.
Since $\Lin\equiv\Sym$ we have
\[
\rtree^{\bullet}=\Lin_+(\rtree)\equiv\Sym_+(\rtree)=\End_+
\]
so $|\rtree^{\bullet}[n]|=|\End[n]|=n^n$ for $n\ge 1$. Finally,
$$
n^n=|\rtree^{\bullet}[n]|=|\tree^{\bullet\bullet}[n]|=n^2\cdot|\tree[n]|
$$
from which Cayley's formula immediately follows.

If we assume that $U=\{u_1,\ldots,u_n\}$ is totally ordered, with
$u_1<u_2<\cdots<u_n$, then Joyal's proof induces a
bijection between doubly-rooted trees on $U$ and functional digraphs on $U$.
In this case, we map the spine~$s\in\Lin[U]$ of a
doubly-rooted tree to the permutation $g\in\Sym[U]$ such that
$g(u_i)=u_j \Leftrightarrow s(i)=u_j$, leaving the rooted trees attached to the spine
untouched. This way we obtain a permutation of rooted trees, that is, an
endofunction.
For instance, the doubly-rooted tree associated with the endofunction of
Figure~\ref{figure_endofun} is depicted in Figure~\ref{figure_joyal}.

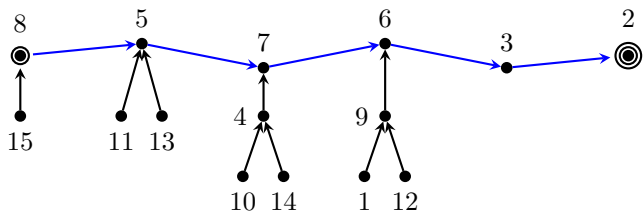
\begin{figure}
\centering
\begin{tikzpicture}
[dot/.style = {fill=black,draw,circle, inner sep=1.25pt, node contents={}, label=#1},
every loop/.style={min distance=7mm,in=50,out=130,looseness=10},
thick,
baseline=0pt,
scale=0.8]
\draw (0,2) circle (4pt);
\draw (10,2) circle (4pt);
\draw (10,2) circle (6pt);
\begin{scope}[font = {\small}]
      \node[outer sep=3pt] (8) at (0,2) [dot=above:$8$];
      \node (5) at (2,2.2) [dot=above:$5$];
      \node (7) at (4,1.8) [dot=above:$7$];
      \node (6) at (6,2.2) [dot=above:$6$];
      \node (3) at (8,1.8) [dot=above:$3$];
      \node[outer sep=5pt] (2) at (10,2) [dot=above:$2$];
      \path (8) edge[blue,-stealth,dashed] (5);
      \path (5) edge[blue,-stealth,dashed] (7);
      \path (7) edge[blue,-stealth,dashed] (6);
      \path (6) edge[blue,-stealth,dashed] (3);
      \path (3) edge[blue,-stealth,dashed] (2);
      \node (15) at (0,1) [dot=below:$15$];
      \node (11) at (1.66,1) [dot=below:$11$];
      \node (13) at (2.33,1) [dot=below:$13$];
      \node (4) at (4,1) [dot=left:$4$]; 
      \node (10) at (3.66,0) [dot=below:$10$];
      \node (14) at (4.33,0) [dot=below:$14$];
      \node (9) at (6,1) [dot=left:$9$];
      \node (1) at (5.66,0) [dot=below:$1$];
      \node (12) at (6.33,0) [dot=below:$12$];
      \path (4) edge[-stealth] (7);
      \path (10) edge[-stealth] (4);
      \path (14) edge[-stealth] (4);
      \path (11) edge[-stealth] (5);
      \path (13) edge[-stealth] (5);     
      \path (9) edge[-stealth] (6);
      \path (1) edge[-stealth] (9);
      \path (12) edge[-stealth] (9);
      \path (15) edge[-stealth] (8);
\end{scope}
\end{tikzpicture}
\caption{The doubly-rooted tree associated with the endofunction of
Figure~\ref{figure_endofun} under Joyal's bijection. The spine is
colored in blue and dashed; tail and head are distinguished with one and
two circles, respectively.}\label{figure_joyal}
\end{figure}

By adjusting equation~\eqref{eq_end=permrtree}, one obtains species
identities for several classes of endofunctions defined in
terms of properties of their functional digraph.
For instance, the recurrent points of a connected endofunction form
a cyclic permutation, and hence the species of \emph{connected
endofunctions} is $\Cyc(\rtree)$. Similarly, if the recurrent points
are structured in a set, we obtain \emph{forests} (acyclic functional
digraphs), $\Set(\rtree)$. If the permutation of recurrent points
is a derangement, we obtain functional digraphs with no fixed points,
$\Der(\rtree)$. In some cases, pairing these equations with Cayley's
formula leads to explicit enumeration of the corresponding digraphs.
These examples are well known; see Table~\ref{table_oeis}
(A-numbers refer to the OEIS~\cite{Sl}).

\begin{table}
\begin{center}
\renewcommand{\arraystretch}{1.25}
\begin{tabular}{p{3cm}rlrlrl}
  & \multicolumn{2}{l}{Permutations} & \multicolumn{2}{l}{Endofunctions} \\
  \toprule
  All structures & $\Sym$ & $n!$                     & $\End$         & $n^n$ \\
  Trees          & $1$    & $\delta_{n,0}$           & $\rtree$       & $n^{n-1}$ \\
  Forests        & $\Set$ & $1$                      & $\Set(\rtree)$ & $(n+1)^{n-1}$ \\
  Connected      & $\Cyc$ & $(n-1)!(1-\delta_{n,0})$ & $\Cyc(\rtree)$ & A001865 \\
  Derangements   & $\Der$ & A000166                  & $\Der(\rtree)$ & $(n-1)^n$
\end{tabular}
\end{center}
\caption{Enumeration of functional digraphs of permutations and
endofunctions, where $\delta$ is the Kronecker delta.
}\label{table_oeis}
\end{table}

\section{Cayley permutations}\label{section_cayperms}

We now define the species of Cayley permutations~\cite{CCEG}.
A \emph{Cayley permutation} on~$U$ is a function $f:U\rightarrow [n]$
where $\Img(f)=[k]$ for some $k\leq n=|U|$.
The transport of structure along a bijection $\sigma: U \to V$ is given by
$$
\Cay[\sigma](f)=f\circ\sigma^{-1}.
$$
Cayley permutations and ballots are combinatorially equal~\cite{CCEG}.
In the special case where $U=[n]$, we have the inclusions
$$
\Sym[n]\subseteq\Cay[n]\subseteq\End[n]
$$
and each of these structures can be identified with their functional digraph.
There is, however, no obvious way to define the functional digraph
of Cayley permutations on an arbitrary finite set~$U$. Furthermore,
the transport of structure on $\Sym$ and $\End$ is conjugation, which
cannot be extended to Cayley permutations due to the requirement
$\Img(f)=[k]$. For these reasons, we now assume
$\LL$-species, defined below.

Let $\LL$ denote the category of finite totally ordered sets with
order-preserving bijections as morphisms. An \emph{$\LL$-species} is
a functor $F:\LL\to\BB$. Since there is a unique order-preserving bijection
between any pair of finite totally ordered sets of the same cardinality,
transport of structure is immaterial in the context of $\LL$-species. We often
assume that the underlying total order is $[n]$ with the standard
order and regard permutations and Cayley permutations as subspecies of
endofunctions. This allows us to identify all these structures with
their functional digraph.
Another consequence of having a unique order-preserving bijection
between totally ordered sets with the same cardinality is that for
$\LL$-species $F=G$ if and only if $F(x)=G(x)$.
For example, $\Lin=\Sym$ as $\LL$-species even though $\Lin\neq\Sym$
as $\BB$-species. In this case, a permutation $f\in\Sym[n]$ is
naturally identified with the linear order $f(1)f(2)\ldots f(n)$
in $\Lin[n]$.

Operations on $\BB$-species extend naturally to $\LL$-species.
New operations also become possible. Given two totally ordered sets
$\ell_1=(U_1, \preceq_1)$ and $\ell_2=(U_2,\preceq_2)$, their
\emph{ordinal sum} $\ell=\ell_1 \oplus \ell_2$ is the totally ordered
set $\ell=(U,\preceq)$ where $U=U_1\sqcup U_2$, $\preceq$ respects
$\preceq_1$ and $\preceq_2$, and all elements of $U_1$ are smaller
than the elements of $U_2$.
The totally ordered sets obtained by adding a new minimum or maximum
element to $\ell$ are denoted by $1\oplus\ell$ and $\ell \oplus 1$,
respectively. Given a totally ordered set $\ell=(U,\preceq)$, the
derivative $F'$ of an $\LL$-species $F$ is defined by
$F'[\ell]=F[1\oplus\ell]$. Additionally, the \emph{integral} of $F$,
denoted $\int_0^X F(T) dT$ or $\int F$, is defined by
\[
\left( \int F \right) [\ell] =
\begin{cases}
  \emptyset &  \text{if }\ell = \emptyset, \\
  F[\ell \setminus \{\min(\ell)\}] & \text{if } \ell \neq \emptyset.
\end{cases}
\]
In other words, an $F'$-structure is an $F$-structure on a total
order endowed with a new minimum, and an $\int\! F$-structure
is an $F$-structure on a total order deprived of its minimum.
Derivative and integral can be alternatively defined in terms of
maximum elements: $F'[\ell]=F[\ell\oplus 1]$ and
$(\int F)[\ell]=F[\ell\setminus \{\max(\ell)\}]$ for $\ell\neq\emptyset$.
Similarly to $\BB$-species, operations on $\LL$-species behave well
with respect to generating series, including a species version of
the fundamental theorem of calculus~\cite[equation 35]{BLL}:
$$
\int F' = F_+;
\qquad
\left(\int F\right)' = F;
\qquad
\left( \int_0^X F(T) dT \right)(x) = \int_0^x F(t) dt.
$$
Let us now return to Cayley permutations.
We have observed in Section~\ref{section_species} that functional
digraphs of endofunctions are precisely those directed graphs where
each node has outdegree one. It is easy to see that a functional digraph
is the functional digraph of a permutation if and only if each node
has both indegree and outdegree equal to one.
The following lemma holds because
the internal nodes of $f\in\End[n]$ are labeled by $\Img(f)$, and we
omit its proof.
It characterizes the functional digraphs of Cayley permutations as those
whose internal nodes are labeled with the smallest
elements of the underlying total order.

\begin{lemma}\label{caygraph}
Let $f\in\End[n]$. Then $f$ is a Cayley permutation with $\Img(f)=[k]$
if and only if the internal nodes in the functional digraph of~$f$
have labels $[k]$.
\end{lemma}

Equivalently, $\Cay$ as an $\LL$-species is isomorphic to the
species of functional digraphs of endofunctions~$f\colon\ell\to\ell$ such that each
internal node is smaller than each leaf.
Although $\LL$-species address some of the cosmetic issues of
Cayley permutations, there are structural problems that require a
more sophisticated approach. As an example, consider the following
functional digraph:
$$
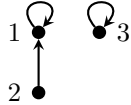
\begin{tikzpicture}
[dot/.style = {draw,circle, inner sep=1.5pt, node contents={}, label=#1},
every loop/.style={min distance=7mm,in=50,out=130,looseness=10},
thick,
baseline=0pt,
scale=0.8]
\begin{scope}[font = {\small}]
\node[fill=black](2) at (0,0)[dot=left:$2$];
\node[fill=black](1) at (0,1)[dot=left:$1$];
\node[fill=black](3) at (1,1)[dot=right:$3$];
\path (1) edge[loop above, -stealth] (1);
\path (3) edge[loop above, -stealth] (3);
\path (2) edge[-stealth] (1);
\end{scope}
\end{tikzpicture}
$$
It is the functional digraph of the endofunction~$f\in\End[3]$
where $f(1)=f(2)=1$ and $f(3)=3$.
In each connected component, every internal node is smaller than
every leaf, yet this property fails for the full functional digraph:
node~$3$ is internal but larger than the leaf~$2$.
Thus $f$ is not a Cayley permutation.
More generally, an endofunction is a permutation of rooted trees,
but a Cayley permutation is \emph{not} a permutation
of rooted trees of Cayley permutations. Similarly, we cannot use
composition to obtain species equations for Cayley permutations whose
functional digraph is connected, a forest, or a derangement.
Generalizing the equation $\Sym=\Set\cdot\Der$ is also not possible:
a Cayley permutation is \emph{not} a pair consisting of the set of
its fixed points together with a Cayley permutation without fixed points.

To circumvent these issues, we assume that internal nodes and leaves are
nodes of different sorts. This can be formalized using two-sort species.

A \emph{two-sort $\BB$-species} (or \emph{$\BB$-species of two sorts})
is a functor $F:\BB\times\BB\to\BB$, where $\BB\times\BB$ denotes
the product category with pairs of finite sets as objects and pairs
of bijections as arrows.
Composition of arrows is componentwise and the identity arrow on the
pair $(U,V)$ is $\id_{(U,V)}=(\id_U,\id_V)$.
Two-sort species allow us to build combinatorial structures whose
underlying set contains elements of two different sorts: those
belonging to~$U$ and those belonging to~$V$.
For a two-sort species~$F$, we write $F=F(X,Y)$ to indicate that the
two sorts are called~$X$ and~$Y$. The letters~$X$ and~$Y$ also serve as
singleton species of the corresponding sort:
\begin{align*}
X[U,V]&=\begin{cases}
\{u\},
& \text{if $U=\{u\}$ and $V=\emptyset$};\\
\emptyset, & \text{otherwise}.
\end{cases}\\[1.5ex]
Y[U,V]&=\begin{cases}
\{v\},
& \text{if $U=\emptyset$ and $V=\{v\}$};\\
\emptyset, & \text{otherwise}.
\end{cases}
\end{align*}
The \emph{generating series} of $F(X,Y)$ is
$$
F(x,y)=
\sum_{i,j\ge 0}|F[i,j]|\frac{x^i}{i!}\frac{y^j}{j!},
$$
where we let $F[i,j]=F[[i],[j]]$.
Most of the operations introduced for unisort species can be naturally
extended to the two-sort context, and definitions can be
found in the usual references~\cite{BLL,Cl}. An example relevant
in the next sections is partial differentiation:
$$
\left(\ddx{F}\right)[U,V]=F[U\cup\{*\},V];
\qquad
\left(\ddy{F}\right)[U,V]=F[U,V\cup\{*\}].
$$
Similarly, we have two pointing operations:
\begin{equation}\label{eq_pointing}
F^{\pointX}=X\cdot\left(\ddx{F}\right);
\qquad
F^{\pointY}=Y\cdot\left(\ddy{F}\right).
\end{equation}
Composition of two-sort species is more delicate and requires
additional care. We define it below in the only case needed in
this paper.
Let $F$ be a unisort species and let $G=G(X,Y)$ be a two-sort species.
Then $F\circ G=(F\circ G)(X,Y)$ is the two-sort species with structures
$$
\left(F\circ G\right)[U,V]=
\bigcup_{\beta\in\Par[U \sqcup V]}F[\beta]\times\prod_{B\in\beta}G[B\cap U,B\cap V].
$$
An $\bigl(F\circ G\bigr)$-structure is thus a generalized
partition on elements of two sorts, where each block carries a
$G$-structure (of two-sorts), and the blocks are structured by the
(unisort) species~$F$.
For example, $\Set\circ(X+Y)=\Set(X)\cdot\Set(Y)$ is the species of
sets of elements of two sorts. A more substantial example opens the
next section.

\section{\texorpdfstring{$R$-recurrent functional digraphs}{R-recurrent functional digraphs}}\label{section_Rrecdig}

Recall from Section~\ref{section_species} that rooted trees of one
sort are defined by the species equation
$$
\rtree(X) \,=\, X\cdot\Set\bigl(\rtree(X)\bigr).
$$
We use the same letter here to denote the two-sort species
$\rtree=\rtree(X,Y)$ of rooted trees with internal nodes of sort~$X$
and leaves of sort~$Y$. Then
\begin{equation}\label{eq_rtree_twosorts}
\rtree(X,Y) \,=\, X\cdot\Set\bigl(\rtree(X,Y)-X+Y\bigr),
\end{equation}
where $\rtree(X,0)=X$ is the tree consisting of a single root (of
sort~$X$). Indeed---as in the unisort case---a rooted tree of
two sorts is obtained by appending a set of rooted trees of two
sorts to the root~$X$. There is one caveat: when we append
a single node~$X$, it becomes a leaf in the resulting tree
and thus its sort changes from~$X$ to~$Y$. This explains the
additional term $-X+Y$ in equation~\eqref{eq_rtree_twosorts}.

In the unisort case, an endofunction is a permutation of rooted trees,
$\End=\Sym\circ\rtree$. More precisely, functional digraphs of
endofunctions are obtained by structuring the roots of a set of rooted
trees in an $\Sym$-structure. Since this construction preserves
internal nodes and leaves, the composition
$
\Sym\circ\rtree(X,Y)
$
yields functional digraphs of two sorts where internal nodes have
sort~$X$ and leaves have sort~$Y$. The same holds for the species of
Table~\ref{table_oeis} if we replace $\Sym$ with $\Set$ (forests,
that is, acyclic functional digraphs), $\Cyc$ (connected functional
digraphs), and $\Der$ (derangements, that is, functional digraphs with
no fixed points).
This brings us to the following definition of $R$-recurrent functional
digraphs of two sorts (see also Figure~\ref{figure_Rstructure}).

\begin{definition}
Let $R$ be a unisort species. The species $\Psi_R=\Psi_R(X,Y)$
of \emph{$R$-recurrent functional digraphs of two sorts} is defined by
\begin{equation}\label{PsiR_def}
\Psi_R(X,Y)=R\circ\rtree(X,Y).
\end{equation}
\end{definition}

In other words, a $\Psi_R$-structure is obtained by putting an
$R$-structure on the roots of a set of rooted trees of two sorts;
in this context, the $R$-structure is called the \emph{recurrent
part} and its points are said to be \emph{recurrent}.
Some significant choices of~$R$ are listed in Table~\ref{table_psiR}.
Note also that $\Psi_X(X,Y)=\rtree(X,Y)$, which gives
\begin{equation}\label{psiR=RcompTree}
\Psi_R = R\circ\Psi_X.
\end{equation}

\begin{figure}
\centering
\begin{tikzpicture}[
dot/.style = {draw,circle,inner sep=1.25pt, node contents={}, label=#1},
every loop/.style={min distance=7mm,in=50,out=130,looseness=10},
thick,
baseline=0pt,
scale=0.8]
\begin{scope}[font = {\small}]
\node[fill=black] (1) at (0.5,3) [dot=above:$1$];
\node[fill=black] (4) at (2,3) [dot=above:$4$];
\node[fill=black] (5) at (3.5,3) [dot=above:$5$];
\node[fill=black] (8) at (5,3) [dot=above:$8$];
\node[fill=black] (9) at (6.5,3) [dot=above:$9$];
\draw[rounded corners] (0, 2.7) rectangle (7, 3.7);
\node[above] at (3.5,3.7){$R$};
\node[fill=black] (10) at (0,2) [dot=left:$10$];
\node (a) at (1,2) [dot=right:$a$];
\node[fill=black] (2) at (0,1) [dot=left:$2$];
\node (d) at (-0.5,0) [dot=left:$d$];
\node (g) at (0.5,0) [dot=right:$g$];
\node[fill=black] (6) at (3.5,2) [dot=left:$6$];
\node[fill=black] (7) at (2.75,1) [dot=left:$7$];
\node (f) at (3.5,1) [dot=below:$f$];
\node[fill=black] (3) at (4.25,1) [dot=right:$3$];
\node (c) at (2.75,0) [dot=below:$c$];
\node (b) at (4.25,0) [dot=below:$b$];
\node (e) at (5,2) [dot=below:$e$];
\path (d) edge[-stealth] (2);
\path (g) edge[-stealth] (2);
\path (2) edge[-stealth] (10);
\path (10) edge[-stealth] (1);
\path (a) edge[-stealth] (1);
\path (c) edge[-stealth] (7);
\path (b) edge[-stealth] (3);
\path (7) edge[-stealth] (6);
\path (f) edge[-stealth] (6);
\path (3) edge[-stealth] (6);
\path (6) edge[-stealth] (5);
\path (e) edge[-stealth] (8);
\end{scope}
\end{tikzpicture}
\caption{An $R$-recurrent functional digraph of two sorts.
Nodes of sort~$X$ (internal nodes) are black and labeled by numbers;
nodes of sort~$Y$ (leaves) are white and labeled by letters.
}\label{figure_Rstructure}
\end{figure}
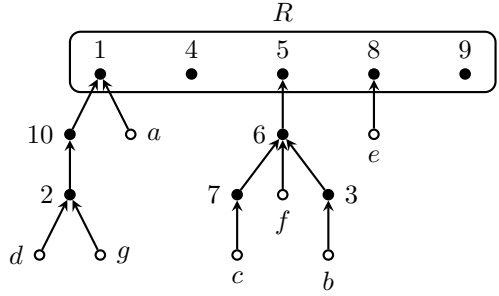

The two-sort species $\Psi_R$ simultaneously encompasses (subspecies
of) Cayley permutations and endofunctions, allowing us to enumerate
functional digraphs of both types at once. From an enumerative
standpoint, using different sorts for internal nodes and leaves
is equivalent to labeling the internal nodes with the smallest
elements of the underlying total order (and the remaining elements
with the largest), which characterizes functional digraphs of Cayley
permutations by Lemma~\ref{caygraph}. On the other hand, using
the same sort $X=Y$ for internal nodes and leaves yields all
functional digraphs, that is, all endofunctions.
We shall make this observation more formal, but first a couple of
definitions.

Let $F=F(X,Y)$ be a two-sort species. The unisort species $F(X,X)$
is obtained by identifying the sorts~$X$ and~$Y$. Note that
\begin{align}
F(X,X)[U] &= \bigcup_{U_1\sqcup U_2=U}F[U_1,U_2]\label{def_FXX}\\
\shortintertext{and}
|F(X,X)[n]|
&=\sum_{i+j=n}\binom{n}{i}|F[i,j]|.\label{eq_FXX}
\end{align}
Assuming $\LL$-species, we also let $\widehat{F}(Z)$ be the unisort
species whose structures are
\begin{equation}\label{def_HATF}
\widehat{F}[\ell]=\bigcup_{\ell_1\oplus \ell_2=\ell}F[\ell_1,\ell_2].
\end{equation}
That is, an $\widehat{F}$-structure on~$\ell$ is simply an
$F(X,Y)$-structure on $[\ell_1,\ell_2]$, where
$\ell=\ell_1\oplus\ell_2$.
As a result,
\begin{equation}\label{eq_HATF}
|\widehat{F}[n]|
=\sum_{i+j=n}|F[i,j]|.
\end{equation}
\begin{table}
\centering
\renewcommand{\arraystretch}{1.25}
\begin{tabular}{llll}
$R$         & $\Psi_R(X,Y)$-structures       & Endofunctions & Cayley permutations \\
\toprule
$\Sym$      & All structures                 & $\End$        & $\Cay$              \\
$X$         & Trees                          & $\End_{X}$    & $\Cay_{X}$          \\
$\Set$      & Forests                        & $\End_{\Set}$ & $\Cay_{\Set}$       \\
$\Cyc$      & Connected                      & $\End_{\Cyc}$ & $\Cay_{\Cyc}$       \\
$\Der\quad$ & Derangements                   & $\End_{\Der}$ & $\Cay_{\Der}$
\end{tabular}
\caption{Unisort species of endofunctions and Cayley permutations
for different choices of the recurrent structure~$R$ in
$\Psi_R(X,Y)$.}\label{table_psiR}
\end{table}

Now, a $\Psi_{\Sym}$-structure where internal nodes and leaves
have the same sort is simply a functional digraph of one sort, and
we have the combinatorial equality
$$
\Psi_{\Sym}(X,X)=\End(X).
$$
On the other hand, by Lemma~\ref{caygraph}
there is a bijection between $\widehat{\Psi}_{\Sym}[n]$
and $\Cay[n]$. More precisely, structures in $\Psi_{\Sym}[i,j]$
correspond to functional digraphs of Cayley permutations on
$[n]=[i+j]$ that have~$i$ internal nodes and~$j$ leaves.
Therefore,
$$
\widehat{\Psi}_{\Sym}(X)=\Cay(X).
$$
This motivates the following definition.
\begin{definition}\label{def_endR_cayR}
  Let $R$ be a unisort species. The species of
  \emph{$R$-recurrent endofunctions} and the species of
  \emph{$R$-recurrent Cayley permutations} are defined as follows:
  \begin{align*}
    \End_R(X) &\,=\, \Psi_R(X,X); \\
    \Cay_R(X) &\,=\, \widehat{\Psi}_R(X).
  \end{align*}
\end{definition}
Our main goal is to find a formula for $|\Psi_R[i,j]|$. Then, by
choosing $R$ as one of $\Sym$, $X$, $\Set$, $\Cyc$, or $\Der$, and letting $F=\Psi_R$
in equations~\eqref{eq_FXX} and~\eqref{eq_HATF}, we obtain
counting formulas for the endofunctions and Cayley permutations
of Table~\ref{table_psiR}.
Comparing~\eqref{eq_FXX} and~\eqref{eq_HATF} also shows that the
number of endofunctions with~$i$ internal nodes and~$j$ leaves is
simply obtained by multiplying the corresponding number of Cayley
permutations by the binomial coefficient $\binom{i+j}{i}$; this
additional term accounts for the possibility of labeling internal
nodes and leaves of endofunctions without the constraint of using
the smallest~$i$ labels for the internal nodes.

In Section~\ref{section_cayperms}, we observed that neither the equation
$\End = \Sym\circ \rtree$ nor the equation $\Sym = \Set\cdot\Der$
readily extends to Cayley permutations. Both constructions are, however,
valid for the two-sort species $\Psi_{\Sym}$. Indeed,
$\Psi_{\Sym}=\Sym\circ \Psi_X$ by definition, and using
$\Sym=\Set\cdot\Der$ it follows that
\begin{align*}
  \Psi_{\Sym} &= (\Set\cdot\Der)\circ\Psi_X \\
             &= \Set(\Psi_X)\cdot\Der(\Psi_X) = \Psi_{\Set}\cdot \Psi_{\Der}.
\end{align*}
Therefore, in the two-sort sense, functional digraphs are permutations
of rooted trees. They are also products of sets of rooted trees (forests)
and functional digraphs with no fixed points.

\section{A differential equation}\label{section_diffeq}

Equation~\eqref{PsiR_def} defines $\Psi_R(X,Y)$ as the composition of $R$
with rooted trees of two sorts. Our goal is a formula for the coefficients
$|\Psi_R[i,j]|$. While the definition is transparent, it does not readily
yield these coefficients. We give an alternative recursive construction
leading to a differential equation from which they can be extracted.

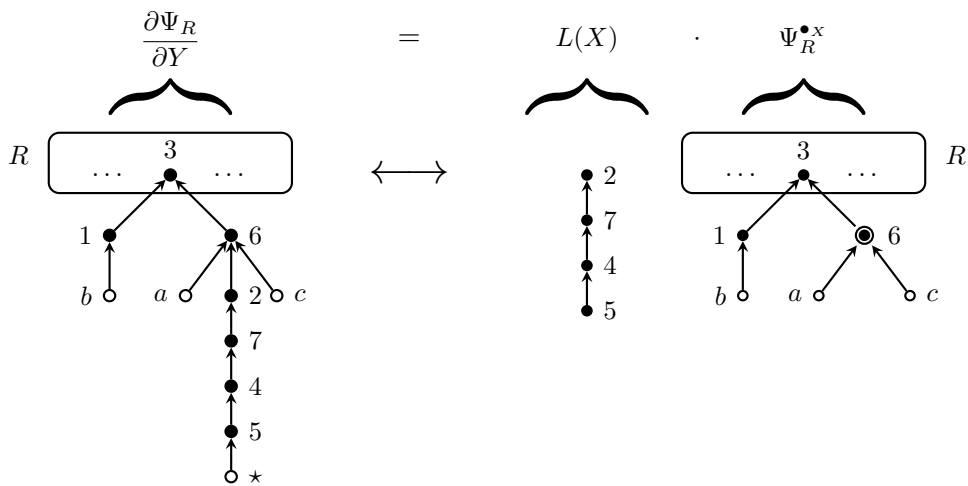
\begin{figure}
\centering
\begin{tikzpicture}[
dot/.style = {draw,circle,inner sep=1.5pt, node contents={}, label=#1},
every loop/.style={min distance=7mm,in=50,out=130,looseness=10},
thick,
baseline=0pt,
scale=0.8]
\begin{scope}[font = {\small}]
\node at (1,5) {$\dots$};
\node at (3,5) {$\dots$};
\node[fill=black] (3) at (2,5) [dot=above:$3$];
\draw[rounded corners] (0,4.7) rectangle (4,5.7);
\node[above] at (-0.5,5){$R$};
\node at (2,7.25){$\displaystyle{\ddy{\Psi_R}}$};
\node[scale=2.5] at (2,6.25){$\overbrace{}$};
\node[fill=black] (1) at (1,4) [dot=left:$1$];
\node (b) at (1,3) [dot=left:$b$];
\node[fill=black] (6) at (3,4) [dot=right:$6$];
\node (a) at (2.25,3) [dot=left:$a$];
\node[fill=black] (2) at (3,3) [dot=right:$2$];
\node (c) at (3.75,3) [dot=right:$c$];
\node[fill=black] (7) at (3,2.25) [dot=right:$7$];
\node[fill=black] (4) at (3,1.5) [dot=right:$4$];
\node[fill=black] (5) at (3,0.75) [dot=right:$5$];
\node (y) at (3,0) [dot=right:$\star$];
\path (b) edge[-stealth] (1);
\path (1) edge[-stealth] (3);
\path (y) edge[-stealth] (5);
\path (5) edge[-stealth] (4);
\path (4) edge[-stealth] (7);
\path (7) edge[-stealth] (2);
\path (2) edge[-stealth] (6);
\path (6) edge[-stealth] (3);
\path (a) edge[-stealth] (6);
\path (c) edge[-stealth] (6);
\end{scope}
\end{tikzpicture}
\quad
\begin{tikzpicture}[
dot/.style = {draw,circle,inner sep=1.25pt, node contents={}, label=#1},
every loop/.style={min distance=7mm,in=50,out=130,looseness=10},
thick,
baseline=0pt,
scale=0.8]
\node at (2,7.25){$\displaystyle{=}$};
\node[scale=1.5] at (2,5){$\longleftrightarrow$};
\end{tikzpicture}
\quad
\begin{tikzpicture}[
dot/.style = {draw,circle,inner sep=1.25pt, node contents={}, label=#1},
every loop/.style={min distance=7mm,in=50,out=130,looseness=10},
thick,
baseline=0pt,
scale=0.8]
\begin{scope}[font = {\small}]
\node at (2,7.25){$\displaystyle{\Lin(X)}$};
\node[scale=2.5] at (2,6.25){$\overbrace{}$};
\node[fill=black] (2) at (2,5) [dot=right:$2$];
\node[fill=black] (7) at (2,4.25) [dot=right:$7$];
\node[fill=black] (4) at (2,3.5) [dot=right:$4$];
\node[fill=black] (5) at (2,2.75) [dot=right:$5$];
\path (5) edge[-stealth] (4);
\path (4) edge[-stealth] (7);
\path (7) edge[-stealth] (2);
\end{scope}
\end{tikzpicture}
\begin{tikzpicture}[
dot/.style = {draw,circle,inner sep=1.25pt, node contents={}, label=#1},
every loop/.style={min distance=7mm,in=50,out=130,looseness=10},
thick,
baseline=0pt,
scale=0.8]
\draw (3,4) circle (4pt);
\begin{scope}[font = {\small}]
\node at (1,5) {$\dots$};
\node at (3,5) {$\dots$};
\node[fill=black] (3) at (2,5) [dot=above:$3$];
\draw[rounded corners] (0,4.7) rectangle (4,5.7);
\node[above] at (4.5,5){$R$};
\node at (0.25,7.25){$\cdot$};
\node at (2,7.25){$\displaystyle{\Psi_R\pointX}$};
\node[scale=2.5] at (2,6.25){$\overbrace{}$};
\node[fill=black] (1) at (1,4) [dot=left:$1$];
\node (b) at (1,3) [dot=left:$b$];
\node[fill=black,outer sep=3pt] (6) at (3,4) [dot=right:$6$];
\node (a) at (2.25,3) [dot=left:$a$];
\node (c) at (3.75,3) [dot=right:$c$];
\path (b) edge[-stealth] (1);
\path (1) edge[-stealth] (3);
\path (6) edge[-stealth] (3);
\path (a) edge[-stealth] (6);
\path (c) edge[-stealth] (6);
\end{scope}
\end{tikzpicture}
\caption{The construction manifested in
  equation~\eqref{eqdiff_psiR}.
  On the left, a $\Psi_R$-structure with an additional leaf,
  denoted by~$\star$.
  On the right, its decomposition into a linear order together with
  a $\Psi_R$-structure in which an internal node is pointed
  (circled).
  The path from~$\star$ through internal nodes of indegree
  one gives the linear order $5\,4\,7\,2$; it attaches
  at~$6$, which becomes the pointed node.
  The boxed region denotes the $R$-structure on the recurrent
  nodes.}\label{figure_rtree_recursive}
\end{figure}

To build any $\Psi_R$-structure, we first build its recurrent part,
an $R$-structure of recurrent points of sort~$X$.
The nonrecurrent points are located on rooted trees of two sorts,
each of which hangs from a recurrent point.
Such a rooted tree can be obtained recursively. Starting from
its root, one repeatedly appends to a specified internal node a path
of internal nodes that ends with a leaf; that is, a structure of
$\Lin(X)\cdot Y$, which we call a \emph{branch}.
This recursive construction of $\Psi_R$-structures is described
by the differential equation
\begin{equation}\label{eqdiff_psiR}
  \ddy{}\Psi_R(X,Y) = \Lin(X)\cdot\Psi_R\pointX(X,Y),\qquad \Psi_R(X,0)=R(X).
\end{equation}
The initial condition determines the recurrent
structure. On the left-hand side of the main equation, we have a
structure of $(\partial/\partial Y)\Psi_R$; that is, a
$\Psi_R$-structure with an additional leaf. Equivalently, consider
the maximal path starting from the additional leaf (excluded) that
consists of internal nodes with indegree at most one. It is an
$\Lin(X)$-structure, which is appended to an internal node of a
$\Psi_R$-structure. By pointing to that internal node we obtain
a $\Psi_R\pointX$-structure.

Assuming $\LL$-species,
we can rewrite equation~\eqref{eqdiff_psiR} using the integral operator as
$$
\Psi_R(X,Y) = R(X) + \int \Lin(X)\cdot\Psi_R\pointX(X,Y)dY.
$$
Figure~\ref{figure_rtree_recursive} illustrates
this construction. We give a formal proof of equation~\eqref{eqdiff_psiR} next.
The proof uses the virtual species identity $\Lin^{-1} = 1-X$;
see~\cite[Section~2.5]{BLL} for background on virtual species.

\begin{theorem}\label{psiX=tree}
  The two-sort species $\Psi_R(X,Y)$ of $R$-recurrent functional
  digraphs satisfies equation~\eqref{eqdiff_psiR}.
\end{theorem}
\begin{proof}
  Recall from equations~\eqref{eq_rtree_twosorts}
  and~\eqref{psiR=RcompTree} that $\Psi_R=R\circ\Psi_X$ and
  $\Psi_X=X\cdot\Set\bigl(\Psi_X-X+Y\bigr)$, where
  $\Psi_X=\rtree(X,Y)$ and $\Psi_X(X,0)=X$. That $\Psi_R$ satisfies
  the initial condition of~\eqref{eqdiff_psiR} is easy to see:
  \[
    \Psi_R(X,0)=R\circ\Psi_X(X,0)=R(X).
  \]
  Next, we calculate the partial derivatives of $\Psi_X$. On differentiating
  the equation $\Psi_X=X\cdot\Set\bigl(\Psi_X-X+Y\bigr)$ we get
  \[
    \ddx{\Psi_X} \,=\, \Bigl(\ddx{\Psi_X} + X - 1\Bigr)\cdot \Set(\Psi_X-X+Y)
  \]
  or, equivalently,
  \[
    X\cdot(1-\Psi_X)\cdot \ddx{\Psi_X} \,=\, (1-X)\cdot \Psi_X.
  \]
  Since $\Lin\cdot (1-X) = 1$ and $\Lin_+ = X\Lin$ we can rewrite this as
  \begin{equation*}
    \Lin_+(X)\cdot \ddx{\Psi_X} = \Lin_+\bigl(\Psi_X\bigr).
  \end{equation*}
  By a similar calculation we find that
  \begin{equation*}
    \ddy{\Psi_X} = \Lin_+\bigl(\Psi_X\bigr)
  \end{equation*}
  and putting these last two equations together we arrive at
  \[
    \ddy{\Psi_X} = \Lin_+(X) \cdot \ddx{\Psi_X}.
  \]
  Finally, we differentiate $\Psi_R=R\circ\Psi_X$ to obtain
  \begin{align*}
    \ddy{\Psi_R} &= \ddy{\Psi_X} \cdot R'(\Psi_X) \\
                 &= \Lin_+(X) \cdot \ddx{\Psi_X} \cdot R'(\Psi_X) \\
                 &= \Lin_+(X) \cdot \ddx{\Psi_R}
                 = \Lin(X)\cdot \Psi_R\pointX,
  \end{align*}
  which concludes the proof.
\end{proof}

Equation~\eqref{eqdiff_psiR} immediately leads to a recursive
formula for $|\Psi_R[i,j]|$.

\begin{lemma}\label{lemma_recursion_psiR}
  We have
  \begin{equation*}
    |\Psi_R[i,j]|=
    \begin{cases}
      |R[i]| & j=0;\\
      0 & j\ge 1,i=0;\\
      i\cdot\bigl(|\Psi_R[i,j-1]|+|\Psi_R[i-1,j]|\bigr) & i,j\ge 1.
    \end{cases}
  \end{equation*}
\end{lemma}
\begin{proof}
  It is easy to see that the first two identities hold.  Next, we
  rewrite equation~\eqref{eqdiff_psiR} of Theorem~\ref{psiX=tree} as
  follows:
  \begin{align*}
    \ddy{\Psi_R} = \Lin\cdot X\cdot \ddx{\Psi_R} 
    &\;\,\iff\;\, \ddy{\Psi_R}\cdot (1-X) = X\ddx{\Psi_R} \\
    &\;\,\iff\;\, \ddy{\Psi_R} = X\cdot \left(\ddx{\Psi_R}+\ddy{\Psi_R}\right).
  \end{align*}
  Identifying coefficients of the last equation gives the third
  identity.
\end{proof}

Next, we use Lemma~\ref{lemma_recursion_psiR} to obtain the
anticipated formula for $|\Psi_R[i,j]|$. It involves the
$r$-Stirling numbers of the second kind, defined as follows.
Let $n,m,r\ge 0$. The \emph{$r$-Stirling number of the second kind}
$\stirl{n}{m}{r}$ is the number of set partitions of $[n]$ with~$m$
blocks such that the numbers $1,2,\ldots,r$ are contained in distinct
blocks. Broder~\cite[Theorem 2]{Br} showed that these numbers obey a
recurrence analogous to the one satisfied by the classic
Stirling numbers:
\begin{equation}\label{def_rstirl}
\stirl{n}{m}{r}=m\stirl{n-1}{m}{r}+\stirl{n-1}{m-1}{r}, \quad r < n,
\end{equation}
with initial conditions $\stirl{n}{m}{n}=\delta_{m,n}$, and $\stirl{n}{m}{r}=0$ if $r>n$.
We define\footnote{We choose this notation mainly to remind us
  of the $r$-Stirling numbers. A fraction also makes some logical sense
  in that a partition is a \emph{division} of an $n$-element set into
  $m$ blocks. The fraction bar may alternatively symbolize the minus
  sign on the right-hand side. Finally, overloading the fraction
  notation has historical precedents in the Legendre and Jacobi
  symbols.}
$$
\sdiff{n}{m}{r}:=\stirl{n}{m}{r}-\stirl{n}{m}{r+1}.
$$
This difference of $r$-Stirling numbers
counts set partitions of $[n]$ with~$m$ blocks such that the numbers
$1,2,\ldots,r$
are contained in distinct blocks, but $1,2,\ldots,r,r+1$ are not.
In other words, $r$ is maximal such that $1,2,\ldots,r$ appear in
distinct blocks. From \eqref{def_rstirl} it follows that
\begin{equation}\label{sdiff_rec}
  \sdiff{n}{m}{r}=m\sdiff{n-1}{m}{r}+\sdiff{n-1}{m-1}{r}, \quad r < n
\end{equation}
except $\sdiff{n}{n}{n-1}=0$, with initial conditions
$\sdiff{n}{m}{n}=\delta_{m,n}$, and $\sdiff{n}{m}{r}=0$ if $r>n$.
In Table~\ref{table_sdiff} we give the numbers $\sdiff{n}{m}{r}$ for
$r = 1,2,3,4$, $n=1,2,\ldots,8$ and $m=1,2,\ldots,8$.

\begin{table}[ht]
\centering
\setlength{\tabcolsep}{3pt}
\begin{tabular}{c@{\hspace{0.5cm}}c}
\begin{tabular}{*{8}{p{1.5em}}}
\multicolumn{8}{c}{$r = 1$} \\
\toprule
 1 &   &   &   &   &   &   &   \\
 1 & 0 &   &   &   &   &   &   \\
 1 & 1 & 0 &   &   &   &   &   \\
 1 & 3 & 1 & 0 &   &   &   &   \\
 1 & 7 & 6 & 1 & 0 &   &   &   \\
 1 & 15 & 25 & 10 & 1 & 0 &   &   \\
 1 & 31 & 90 & 65 & 15 & 1 & 0 &   \\
 1 & 63 & 301 & 350 & 140 & 21 & 1 & 0 \\
\bottomrule
\end{tabular}
&
\begin{tabular}{*{8}{p{1.5em}}}
\multicolumn{8}{c}{$r = 2$} \\
\toprule
 0 &   &   &   &   &   &   &   \\
 0 & 1 &   &   &   &   &   &   \\
 0 & 2 & 0 &   &   &   &   &   \\
 0 & 4 & 2 & 0 &   &   &   &   \\
 0 & 8 & 10 & 2 & 0 &   &   &   \\
 0 & 16 & 38 & 18 & 2 & 0 &   &   \\
 0 & 32 & 130 & 110 & 28 & 2 & 0 &   \\
 0 & 64 & 422 & 570 & 250 & 40 & 2 & 0 \\
\bottomrule
\end{tabular}
\vspace{0.1cm}\\
\begin{tabular}{*{8}{p{1.5em}}}
\multicolumn{8}{c}{$r = 3$} \\
\toprule
 0 &   &   &   &   &   &   &   \\
 0 & 0 &   &   &   &   &   &   \\
 0 & 0 & 1 &   &   &   &   &   \\
 0 & 0 & 3 & 0 &   &   &   &   \\
 0 & 0 & 9 & 3 & 0 &   &   &   \\
 0 & 0 & 27 & 21 & 3 & 0 &   &   \\
 0 & 0 & 81 & 111 & 36 & 3 & 0 &   \\
 0 & 0 & 243 & 525 & 291 & 54 & 3 & 0 \\
\bottomrule
\end{tabular}
&
\begin{tabular}{*{8}{p{1.5em}}}
\multicolumn{8}{c}{$r=4$} \\
\toprule
 0 &   &   &   &   &   &   &   \\
 0 & 0 &   &   &   &   &   &   \\
 0 & 0 & 0 &   &   &   &   &   \\
 0 & 0 & 0 & 1 &   &   &   &   \\
 0 & 0 & 0 & 4 & 0 &   &   &   \\
 0 & 0 & 0 & 16 & 4 & 0 &   &   \\
 0 & 0 & 0 & 64 & 36 & 4 & 0 &   \\
 0 & 0 & 0 & 256 & 244 & 60 & 4 & 0 \\
\bottomrule
\end{tabular}
\end{tabular}
\caption{Triangles of the numbers $\sdiff{n}{k}{r}$,
for $r=1,2,3,4$.}\label{table_sdiff}
\end{table}

\begin{theorem}\label{thm_PsiRr}
For $i,j,r\ge 0$, we have
\begin{equation*}
|\Psi_{R_r}[i,j]|=\frac{i!|R[r]|}{r!}\sdiff{i+j}{i}{r}.
\end{equation*}
\end{theorem}
\begin{proof}
We shall prove that the term
on the right-hand side satisfies the recursion of
Lemma~\ref{lemma_recursion_psiR} with $R=R_r$.
The initial conditions are easy to verify, so let us assume that $i,j\ge 1$.
We need to prove that
\[
\frac{i!|R[r]|}{r!}\sdiff{i+j}{i}{r} =\,
i\cdot\left(
\frac{(i-1)!|R[r]|}{r!}\sdiff{i+j-1}{i-1}{r}
\;+\;\frac{i!|R[r]|}{r!}\sdiff{i+j-1}{i}{r}
\right).
\]
But this is equivalent to
$\sdiff{i+j}{i}{r} = \sdiff{i+j-1}{i-1}{r}+i\sdiff{i+j-1}{i}{r}$,
which follows from~\eqref{sdiff_rec}.
\end{proof}

Using Theorem~\ref{thm_PsiRr} and summing over~$r$ we arrive at the
following corollary.

\begin{corollary}
  For $i,j\ge 0$, we have
  \begin{align}
    |\Psi_R[i,j]|
    &= i!\sum_{r=0}^i\frac{|R[r]|}{r!}\sdiff{i+j}{i}{r}.\label{formula_psiR}
  \end{align}
\end{corollary}

Recall from Definition~\ref{def_endR_cayR} that $\End_R(X)=\Psi_R(X,X)$
and $\Cay_R(X)=\widehat{\Psi}_R(X)$.  The number of $R$-recurrent
endofunctions and Cayley permutations on $[n]=[i+j]$ are now obtained by
combining equation~\eqref{formula_psiR} with identities~\eqref{eq_FXX}
and~\eqref{eq_HATF}.

\begin{corollary}\label{cor_endcay_enum}
For $n\ge 0$, we have
\begin{align*}
|\End_{R}[n]|
    &= \sum_{r=0}^n\frac{|R[r]|}{r!}\sum_{i=0}^n\ff{n}{i}\sdiff{n}{i}{r} 
    \shortintertext{and}
    |\Cay_{R}[n]|
    &= \sum_{r=0}^n\frac{|R[r]|}{r!}\sum_{i=0}^ni!\sdiff{n}{i}{r}. 
\end{align*}
\end{corollary}

The expression $\ff{n}{i}$, above, is the \emph{falling factorial}; it
is defined by
\[
  \ff{n}{i} = n(n-1)(n-2)\cdots(n-i+1) = \binom{n}{i}i!.
\]

Letting $R=\Set_r$ in Corollary~\ref{cor_endcay_enum}, we obtain
two identities for the number of forests of $r$ rooted trees of
endofunctions and Cayley permutations:
\begin{align*}
  |\End_{\Set_r}[n]| &= \sum_{i=0}^n\frac{\ff{n}{i}}{r!}\sdiff{n}{i}{r}; \\
  |\Cay_{\Set_r}[n]| &= \sum_{i=0}^n\frac{i!}{r!}\sdiff{n}{i}{r}.
\end{align*}
These allow us to rewrite the formulas of Corollary~\ref{cor_endcay_enum} as
\begin{align}
|\End_{R}[n]| &= \sum_{r=0}^n|R[r]|\cdot|\End_{\Set_r}[n]|;\label{eq_endonew}\\
|\Cay_{R}[n]| &= \sum_{r=0}^n|R[r]|\cdot|\Cay_{\Set_r}[n]|\label{eq_caynew}.
\end{align}
The combinatorics underlying equations~\eqref{eq_endonew}
and~\eqref{eq_caynew} is clear: An $\End_R$-structure
is obtained by putting an $R$-structure over a set of rooted trees.
The same holds for $\Cay_R$-structures.
Furthermore, a forest of endofunctions is a set of rooted trees:
$\End_{\Set}=\Set\circ\End_X$
and $\End_{\Set_r} = \Set_r\circ\End_X = \frac{1}{r!}(\End_X)^r$.
Thus, hidden underneath equation~\eqref{eq_endonew} we find
the species composition $\End_R = R\circ\End_X$, which
is equation~\eqref{PsiR_def} when $Y=X$.
On the other hand, $\Cay_{\Set}\neq\Set\circ\Cay_X$.

Alternatively, letting $R=\Sym_r$ we obtain
\begin{align*}
  |\End_{R}[n]| &= \sum_{r=0}^n\frac{|R[r]|}{r!}|\End_{S_r}[n]|; \\
  |\Cay_{R}[n]| &= \sum_{r=0}^n\frac{|R[r]|}{r!}|\Cay_{S_r}[n]|,
\end{align*}
where $|\End_{S_r}[n]|$ and $|\Cay_{S_r}[n]|$ are simply the number of
endofunctions and Cayley permutations with $r$ recurrent points.


By substituting the number $|R[r]|$ of $R$-structures of size~$r$ into
Theorem~\ref{thm_PsiRr}, we obtain counting formulas for
the two-sort version of the species of Table~\ref{table_psiR}.
Using Corollary~\ref{cor_endcay_enum}
we also obtain identities for the corresponding unisort structures
listed in Table~\ref{table_formulas}. A noteworthy example is a
formula for the number of Cayley-derangements, the main result
of this paper:
\begin{theorem}\label{thm_eq_cayder}
  The number of fixed-point-free Cayley permutations of $[n]$ is
  \begin{equation*}
    |\Cay_{\Der}[n]|
    \,=\, \sum_{r=0}^n\frac{|\Der[r]|}{r!}\sum_{i=0}^ni!\sdiff{n}{i}{r}.
  \end{equation*}
\end{theorem}
The first few numbers of this sequence are
$$
1,
 0,
 1,
 4,
 25,
 184,
 1617,
 16492,
 191721,
 2503040,
 36267393,
 577560596.
$$
At the time of writing, this sequence is not in the OEIS~\cite{Sl}.


We highlight three additional noteworthy instances of
Corollary~\ref{cor_endcay_enum} below.

\begin{proposition}\label{prop_Cay_new}
  The number of Cayley permutations of $[n]$
  whose functional digraph is a tree equals the number of all
  Cayley permutations of $[n-1]$. That is,
  \[
    |\Cay_{X}[n]| = |\Cay[n-1]|.
  \]
  The number of Cayley permutations of $[n]$ whose functional digraph is
  a forest equals
  \begin{align*}
    |\Cay_E[n]| &= \sum_{r=0}^n\frac{1}{r!}\sum_{i=0}^ni!\sdiff{n}{i}{r}.\\
    \intertext{The number of Cayley permutations of $[n]$ whose functional
    digraph is connected equals}
    |\Cay_{\Cyc}[n]| &= \sum_{r=1}^n\frac{1}{r}\sum_{i=0}^ni!\sdiff{n}{i}{r}.
  \end{align*}
\end{proposition}
\begin{proof}
  Setting $R=X$ in Corollary~\ref{cor_endcay_enum} we get the first identity:
  \[
    |\Cay_{X}[n]|
    = \sum_{r=0}^n\frac{|X[r]|}{r!}\sum_{i=0}^ni!\sdiff{n}{i}{r}
    = \sum_{i=0}^ni!\sdiff{n}{i}{1}
    = |\Cay[n-1]|,
  \]
  where in the last step $\sdiff{n}{i}{1}$ equals the
  $(n-1,i)$th Stirling number of the second kind. The second and third
  identities are immediate consequences of Corollary~\ref{cor_endcay_enum}.
\end{proof}

In Section~\ref{section_Joyal} we generalize Joyal's bijection
between doubly-rooted trees and endofunctions to the two-sort
setting.

\begin{table}
\centering
\renewcommand{\arraystretch}{1.7}
\begin{tabular}{lllcc}
                 & $R$~~~ & $|\Psi_{R_r}[i,j]|$                    & $|\End_R[n]|$     & $|\Cay_R[n]|$ \\
\toprule
All structures   & $\Sym$ & $i!\sdiff{i+j}{i}{r}$                      & $n^n$             & $|\Cay[n]|$ \\
Trees            & $X$    & $i!\delta_{r,1}\sdiff{i+j}{i}{r}$          & $n^{n-1}$         & $|\Cay[n-1]|$ \\
Forests          & $\Set$ & $\frac{i!}{r!}\sdiff{i+j}{i}{r}$           & $n^{n-2}$         & Proposition~\ref{prop_Cay_new} \\
Connected        & $\Cyc$ & $\frac{i!}{r}\sdiff{i+j}{i}{r}$            & A001865           & Proposition~\ref{prop_Cay_new} \\
Derangements     & $\Der$ & $|\Der[r]|\frac{i!}{r!}\sdiff{i+j}{i}{r}$  & $(n-1)^n$         & Theorem~\ref{thm_eq_cayder} \\
\end{tabular}
\caption{Number of $R$-recurrent functional digraphs with~$r$
recurrent points, $i$ internal nodes and $j$ leaves.
}\label{table_formulas}
\end{table}

\section{Asymptotic proportion of Cayley-derangements}\label{section_asymp}

We opened this paper by recalling the answer to the hat-check
problem: the probability that a permutation chosen uniformly at
random is a derangement converges to $1/e$ as its size grows to
infinity. In this section we prove that Cayley-derangements have the
same asymptotic behaviour.

\begin{theorem}\label{main-theorem}
Let $\fub{n} = |\Cay[n]|$ denote the $n$th Fubini number. Then
\[
  \lim_{n\to\infty} \frac{|\CayDer[n]|}{\fub{n}} = \frac{1}{e}.
\]
\end{theorem}

Define
\[
  \rfub{n}{r} = \sum_{i \geq 0} i!\, \stirl{n}{i}{r}
  \quad\text{and}\quad
  \sfub{n}{r} = \sum_{i \geq 0} i!\, \sdiff{n}{i}{r}.
\]
In particular, $\sfub{n}{r} = \rfub{n}{r} - \rfub{n}{r+1}$,
$\rfub{n}{0}=\fub{n}$ and $\sfub{n}{0} = 0$. Next, define the weights
\[
  w_r = \frac{|\Der[r]|}{r!} = \sum_{k=0}^{r} \frac{(-1)^k}{k!},
\]
so that Theorem~\ref{thm_eq_cayder} takes the form
\[
  |\CayDer[n]| = \sum_{r\ge 0} w_r\,\sfub{n}{r}.
\]
Define the error $\epsilon_r = w_r - 1/e$.
\begin{lemma}\label{lem:dererr}
  We have
  \[
    \epsilon_r = \frac{(-1)^r}{(r+1)!}
    + O\!\left(\frac{1}{(r+2)!}\right).
  \]
  In particular, $|\epsilon_r| \leq 1/(r+1)!$.
\end{lemma}
\begin{proof}
  Since $\epsilon_r = -\sum_{k \geq r+1} (-1)^k/k!$ is an alternating
  series with decreasing terms, the result follows.
\end{proof}

Denote by\vspace{-1ex}
\begin{align}
  p_r(n) &= \frac{\sfub{n}{r}}{\fub{n}}\nonumber \\
  \intertext{the proportion of ballots in which $r$ is maximal
  such that $1,\ldots,r$ are in distinct blocks. Then}
  \frac{|\CayDer[n]|}{\fub{n}} &= \sum_{r \geq 0} w_r\, p_r(n). \label{eq:prop_of_cayder}
\end{align}

\begin{lemma}\label{lem:fubini}
For all $n \geq 0$,
$p_r(n)$ defines a probability distribution in~$r$
on~$\{0, 1, \ldots, n\}$. That is,
$\sfub{n}{r} \geq 0$ for all $0\le r\le n$ and
$\sum_{r=0}^{n} \sfub{n}{r} = \fub{n}$.
\end{lemma}

\begin{proof}
  Every ballot of~$[n]$ has a unique maximal~$r$ such that
  $1, \ldots, r$ are in distinct blocks and hence
  $\sum_{r=0}^{n} \sfub{n}{r} = \fub{n}$. Further, $\sfub{n}{r}$ is
  trivially non-negative.
\end{proof}

\begin{lemma}\label{lem:error}
For all $n \geq 0$,
\[
  \frac{|\CayDer[n]|}{\fub{n}} - \frac{1}{e}
  = \sum_{r \geq 0} \left(\frac{|\Der[r]|}{r!} - \frac{1}{e}\right) p_r(n)
  = \sum_{r \geq 0} \epsilon_r\, p_r(n).
\]
\end{lemma}

\begin{proof}
Writing $w_r = 1/e + \epsilon_r$ in
equation~\eqref{eq:prop_of_cayder} gives the identity.
\end{proof}

For $n \geq 2$ and distinct $a, b \in [n]$, let $S(a,b;n,i)$
denote the set of partitions of~$[n]$ into $i$~blocks in which
$a$ and~$b$ lie in the same block. For instance,
the set of partitions of~$[4]=\{1,2,3,4\}$ into two blocks
in which $1$ and~$2$ lie in the same block is
\[
  S(1,2;4,2)
  = \bigl\{\, \{1,2\}\{3,4\},\;
    \{1,2,3\}\{4\},\; \{1,2,4\}\{3\} \,\bigr\}
\]
and by deleting $2$ from its block we obtain all set partitions
of $\{1,3,4\}$ into two blocks:
\[
  \bigl\{\, \{1\}\{3,4\},\;
    \{1,3\}\{4\},\; \{1,4\}\{3\} \,\bigr\}.
\]
This process is clearly invertible, and more generally we obtain
the following simple lemma.

\begin{lemma}\label{lem:merge}
For $n \geq 2$ and distinct $a, b \in [n]$, set partitions
of~$[n]$ into $i$~blocks in which $a$ and~$b$ lie in the
same block are in one-to-one correspondence with set partitions
of $[n]\setminus\{b\}$ into $i$ blocks. In particular,
\[
  \bigl|S(a,b;n,i)\bigr| = \setstirling{n-1}{i}.
\]
\end{lemma}

To bound the tail $\sum_{r=0}^{K-1} p_r(n)$ that appears in
the proof of Theorem~\ref{main-theorem}, we need an upper bound
on~$\fub{n} - \rfub{n}{K}$.
Lemma~\ref{lem:merge} provides one by summing over the
$\binom{K}{2}$ pairs in~$[K]$.

\begin{lemma}\label{lem:union}
For $2 \leq K \leq n$,
\[
  \fub{n} - \rfub{n}{K} \leq \binom{K}{2}\, \fub{n-1}.
\]
\end{lemma}

\begin{proof}
The difference $\setstirling{n}{i} - \stirl{n}{i}{K}$ counts
partitions of~$[n]$ into $i$~blocks in which at least one pair
from~$[K]$ shares a block. Thus,
\begin{align*}
  \setstirling{n}{i} - \stirl{n}{i}{K}
  &= \Biggl|\,\bigcup_{\{a,b\} \subseteq [K]}S(a,b;n,i)\,\Biggr|\\[1ex]
  &\leq \sum_{\{a,b\} \subseteq [K]} |S(a,b;n,i)|
  \leq \binom{K}{2}\, \setstirling{n-1}{i}.
\end{align*}
Multiplying by $i!$ and summing over $i$ yields
$\fub{n} - \rfub{n}{K}
\leq \binom{K}{2}\, \fub{n-1}$.
\end{proof}

The asymptotic behaviour of the Fubini numbers is well known and
readily obtainable from the generating series
$\sum_{n \geq 0} \fub{n}\, x^n/n! = 1/(2 - e^x)$
by singularity analysis~\cite{FS}.

\begin{lemma}\label{lem:fubini-asymp}
We have
\[
  \fub{n}
  = \frac{n!}{2(\log 2)^{n+1}}
    \bigl(1+O(r_1^{\,n})\bigr),
  \quad\text{where}\quad
  r_1
  = \frac{\log 2}{\sqrt{(\log 2)^2+(2\pi)^2}} <1.
\]
In particular, $\fub{n-1}/\fub{n} \sim (\log 2)/n \to 0$.
\end{lemma}

\begin{proof}[Proof of Theorem~\ref{main-theorem}]
By Lemma~\ref{lem:error},
\[
  \frac{|\CayDer[n]|}{\fub{n}} - \frac{1}{e}
  = \sum_{r \geq 0} \epsilon_r\, p_r(n),
\]
where $(p_r(n))_{r \geq 0}$ is the probability distribution
of Lemma~\ref{lem:fubini} and
$|\epsilon_r| \leq 1/(r+1)!$ (Lemma~\ref{lem:dererr}).
Fix an integer $K \geq 2$ and apply the triangle inequality:
\[
  \left|\sum_{r \geq 0} \epsilon_r\, p_r(n)\right|
  \leq \sum_{r=0}^{K-1} |\epsilon_r|\, p_r(n)
  + \sum_{r \geq K} |\epsilon_r|\, p_r(n).
\]
Bounding $|\epsilon_r| \leq 1$ in the first sum and
$|\epsilon_r| \leq 1/(K+1)!$ in the second gives
\[
  \left|\sum_{r \geq 0} \epsilon_r\, p_r(n)\right|
  \leq \sum_{r=0}^{K-1} p_r(n) + \frac{1}{(K+1)!}.
\]
Since $\sfub{n}{r} = \rfub{n}{r} - \rfub{n}{r+1}$,
the sum telescopes:
\begin{align*}
  \sum_{r=0}^{K-1} p_r(n)
  &= \frac{1}{\fub{n}}
     \sum_{r=0}^{K-1}
     \bigl(\rfub{n}{r} - \rfub{n}{r+1}\bigr) \\
  &= \frac{\rfub{n}{0} - \rfub{n}{K}}{\fub{n}}
  = \frac{\fub{n} - \rfub{n}{K}}{\fub{n}}.
\end{align*}
Applying Lemma~\ref{lem:union} yields
\[
  \left|\frac{|\CayDer[n]|}{\fub{n}} - \frac{1}{e}\right|
  \leq \binom{K}{2}\frac{\fub{n-1}}{\fub{n}}
  + \frac{1}{(K+1)!}.
\]
By Lemma~\ref{lem:fubini-asymp},
$\fub{n-1}/\fub{n} \to 0$.
Given $\varepsilon > 0$, first choose $K$ so that
$1/(K+1)! < \varepsilon/2$, then choose $N$ so that
$\binom{K}{2} \fub{n-1}/\fub{n} < \varepsilon/2$
for all $n \geq N$. The result follows.
\end{proof}

\begin{proposition}\label{prop:cayder-asymp}
The rate of convergence in Theorem~\ref{main-theorem} is
\[
  \frac{|\CayDer[n]|}{\fub{n}} - \frac{1}{e} \sim -\frac{\log 2}{2en}.
\]
\end{proposition}

The proof is similar to that of Theorem~\ref{main-theorem}
but more involved; see Appendix~\ref{appendix_asymp}.

For comparison, recall from Lemma~\ref{lem:dererr} that the analogous
error for permutations is
\[
  \frac{|\Der[n]|}{n!} - \frac{1}{e} \sim \frac{(-1)^{n}}{(n+1)!},
\]
which decays super-exponentially, while for endofunctions the ratio is
exactly $(1-1/n)^n$ and
\[
  \left(1-\frac{1}{n}\right)^n - \frac{1}{e} \sim -\frac{1}{2en}.
\]
The Cayley-derangement and endofunction errors are both eventually negative, and
since $\log 2 < 1$, the former is eventually smaller in magnitude. The permutation
error is negligible compared to either. This gives the ordering
\[
  \frac{|\Der[n]|}{n!} > \frac{|\CayDer[n]|}{\fub{n}} > \left(1-\frac{1}{n}\right)^n
\]
for large~$n$.

\section{A two-sort species extension of Joyal's bijection}\label{section_Joyal}

Recall from Section~\ref{section_species} Joyal's
bijection between doubly-rooted trees and functional digraphs of
endofunctions. It identifies the spine of a doubly-rooted tree with
a nonempty permutation, leaving the rooted trees that hang from the
spine untouched. This bijection implies the unisort
$\LL$-species identity
\begin{equation}\label{eq_joyalunisort}
\rtree^{\bullet} = \End_+.
\end{equation}
We shall extend Joyal's construction to the two-sort case by
revealing a link between functional digraphs $\Psi_{\Sym}$
and rooted trees $\Psi_X=\rtree(X,Y)$.
In the proof of Theorem~\ref{psiX=tree}, we showed that
$$
\ddy{\Psi_X} = \Lin_+\bigl(\Psi_X\bigr).
$$
Assuming $\LL$-species, we have $\Lin=\Sym$ and hence
$\Lin_+(\rtree) = \Sym_+(\rtree)=\Psi_{\Sym_+}$. That is,
\begin{equation}\label{eq_joyal2}
\ddy{\Psi_X}=\Psi_{\Sym_+}.
\end{equation}
We have found a combinatorial proof of~\eqref{eq_joyal2}
in the spirit of Joyal's proof of~\eqref{eq_joyalunisort}.
On the left-hand side, we have a rooted tree
of two sorts with an additional leaf. Define its \emph{spine} as the
path from this leaf to the root. Ignoring the initial
leaf, this path is a nonempty linear order of (internal) nodes of
sort~$X$, where a rooted tree of two-sorts hangs from each of its
nodes. Identifying the spine with a permutation in the usual way
yields a nonempty permutation of rooted trees, that is, a structure of
$\Sym_+\circ\Psi_X=\Psi_{\Sym_+}$, the species on the right-hand side.
In other words, the additional leaf of a structure of
$(\partial/\partial Y)\Psi_X$ plays the role of \emph{tail}, while
the root plays the role of \emph{head}.
A two-sort vertebrate is born in the process, and by reading a
permutation off its spine we obtain a permutation of rooted trees of
two sorts; that is, a structure of $\Psi_{\Sym_+}$.
This construction---an instance of which can be
found in Figure~\ref{figure_2sortJoyal}---is easily invertible,
and thus it yields a combinatorial proof of equation~\eqref{eq_joyal2}.

\begin{figure}
\centering
\begin{tikzpicture}[
dot/.style = {draw,circle,inner sep=1.5pt, node contents={}, label=#1},
every loop/.style={min distance=7mm,in=50,out=130,looseness=10},
thick,
baseline=0pt,
scale=0.8]
\begin{scope}[font = {\small}]
\node (y) at (0.5,3) [dot=above:$y^\star$];
\node[fill=black] (5) at (2,3) [dot=above:$5$];
\node[fill=black] (2) at (3.5,3) [dot=above:$2$];
\node[fill=black] (4) at (5,3) [dot=above:$4$];
\node[fill=black] (3) at (6.5,3) [dot=right:$3$];
\node[fill=black] (1) at (2,2) [dot=left:$1$];
\node (a) at (1.625,1) [dot=left:$a$];
\node (b) at (2.375,1) [dot=right:$b$];
\node (c) at (4.625,2) [dot=left:$c$];
\node (e) at (5.375,2) [dot=right:$e$];
\node (d) at (6.5,2) [dot=below:$d$];
\path (y) edge[blue,-stealth,dashed] (5);
\path (5) edge[blue,-stealth,dashed] (2);
\path (2) edge[blue,-stealth,dashed] (4);
\path (4) edge[blue,-stealth,dashed] (3);
\path (3) edge[loop, -stealth] (3);
\path (a) edge[-stealth] (1);
\path (b) edge[-stealth] (1);
\path (1) edge[-stealth] (5);
\path (c) edge[-stealth] (4);
\path (e) edge[-stealth] (4);
\path (d) edge[-stealth] (3);
\end{scope}
\end{tikzpicture}
\hspace{1.5cm}
\begin{tikzpicture}[
dot/.style = {draw,circle,inner sep=1.5pt, node contents={}, label=#1},
every loop/.style={min distance=7mm,in=50,out=130,looseness=10},
thick,
baseline=0pt,
scale=0.8]
\begin{scope}[font = {\small}]
      \node[fill=black] (2) at (0.5,3) [dot=above:$2$];
      \node[fill=black] (3) at (0,2.25) [dot=left:$3$];
      \node[fill=black] (5) at (1,2.25) [dot=right:$5$];
      \node (d) at (0,1.25) [dot=left:$d$];
      \node (a) at (0.66,0.25) [dot=below:$a$];
      \node (b) at (1.33,0.25) [dot=below:$b$];
      \node[fill=black] (1) at (1,1.25) [dot=right:$1$];
      \path (3) edge[-stealth, bend left] (2);
      \path (2) edge[-stealth, bend left] (5);
      \path (5) edge[-stealth, bend left] (3);
      \path (d) edge[-stealth] (3);
      \path (a) edge[-stealth] (1);
      \path (b) edge[-stealth] (1);
      \path (1) edge[-stealth] (5);
      \node[fill=black] (4) at (2.5,3) [dot=left:$4$];
      \node (c) at (2.125,2) [dot=below:$c$];
      \node (e) at (2.875,2) [dot=below:$e$];
      \path (4) edge[loop above, -stealth] (4);
      \path (c) edge[-stealth] (4);
      \path (e) edge[-stealth] (4);
\end{scope}
\end{tikzpicture}
\caption{A rooted tree of two sorts with an additional leaf,
denoted $y^\star$, and the corresponding functional digraph under
the extension of Joyal's construction. The spine of the tree,
dashed in blue, is the linear order $5243$; it is mapped to the recurrent
part of the functional digraph on the right.
}\label{figure_2sortJoyal}
\end{figure}
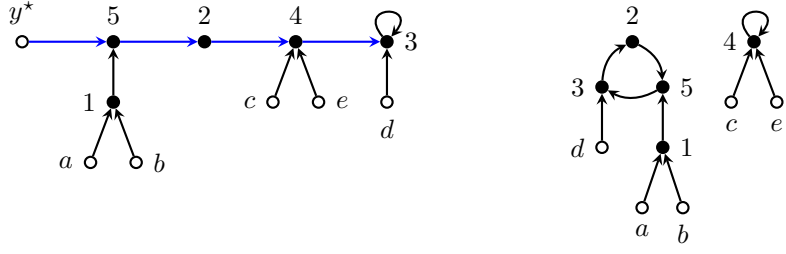

\section{Unisort species of Cayley functional digraphs}\label{section_unisort}

In this section, we build on the enumerative results obtained
for $\Psi_R$ in the two-sort case to obtain species
equations for its unisort counterpart $\Cay_R$.
These rely on the \emph{ordinal product} of $\LL$-species, defined as
\[
(F \odot G)[\ell] \,= \sum_{\ell=\ell_1 \oplus \ell_2} F[\ell_1] \times G[\ell_2].
\]
An $F\odot G$- structure is thus obtained by putting an $F$-structure
on the initial segment $\ell_1$ and a $G$-structure on the terminal
segment $\ell_2$.

We shall use $\LL$-species to show that
$$
\Cay_R
=\frac{1}{2}\sum_{r\ge 0} R_r\odot(1+\Set^r\Bal^r).
$$
This specializes to give expressions for the unisort species of
Cayley-derangements and the other species in Table~\ref{table_psiR}:
\begin{equation}\label{unisort_cayder}
\Cay_{\Der}
=\frac{1}{2}\sum_{r\ge 0} \Der_r\odot\left(1+\Set^r\Bal^r\right).
\end{equation}

Let us start with a simple lemma.

\begin{lemma}\label{lemma_EBal}
For $r\ge 1$, we have
\begin{equation*}
\sum_{i\ge 0}i!\cdot\sdiff{n+r+1}{i}{r}
=r!\cdot r\cdot |\bigl(\Set^r\Bal^{r+1}\bigr)[n]|.
\end{equation*}
\end{lemma}
\begin{proof}
The left-hand side counts ballots of $[n+r+1]$ with~$i$ blocks,
where~$r$ is maximal such that $1,\dots,r$ are contained in
distinct blocks. Any such ballot is structured
$$
\beta_1\;\{u_1,\ldots \}\;
\beta_2\;\{u_2,\ldots \}\;
\dots\;
\beta_r\;\{u_r,\ldots\}\;
\beta_{r+1},
$$
where $u_1u_2\dots u_r$ is a permutation of $1,\ldots,r$
and the $\beta_i$'s are ballots filling the space between the blocks
containing the $u_i$'s.
The permutation of $1,\ldots,r$ gives the $r!$ term
on the right-hand side. Since $r+1$ belongs to one of the blocks containing the $u_i$'s, we have~$r$ choices to arrange it.
Finally, the remaining elements, $r+2,r+3,\ldots,n+r+1$, are
arranged in a structure of $\Set^r\cdot\Bal^{r+1}$ (of size $n$):
elements in the same block as some $u_j$ determine an
$\Set^r$-structure, while the other elements determine the ballots
$\beta_1,\dots,\beta_{r+1}$, that is, a structure of $\Bal^{r+1}$.
\end{proof}

\begin{theorem}\label{thm_uniPsiR}
We have the $\LL$-species identity
\[
  \Cay_{R} = R + \sum_{r\ge 1}
  (R_r)^{\bullet}\odot\left(\int\Set^r\Bal^{r+1}\right). 
\]
\end{theorem}
\begin{proof}
Let $F$ denote the species on the right-hand side.
To prove the desired $\LL$-species identity, it suffices to show
that $|F[n]|=|\Cay_{R}[n]|$ for $n\ge 0$.
Expanding the definition of~$F$ yields
\begin{align*}
|F[n]|
&= |R[n]| + \sum_{r=1}^n
|R^{\bullet}[r]|\cdot
\left|\left(\int \Set^r\Bal^{r+1}\right)[n-r]\right|\\
&= |R[n]| + \sum_{r=1}^{n-1}
r|R[r]|\cdot
\left|\left(\Set^r\Bal^{r+1}\right)[n-r-1]\right|\cdot\frac{r!}{r!}\\
&= |R[n]| + \sum_{r=1}^{n-1}\frac{|R[r]|}{r!}\sum_{i\ge 0}i!\sdiff{n}{i}{r}
&& \text{(by Lemma~\ref{lemma_EBal})}\\
&= \sum_{r=1}^{n}\frac{|R[r]|}{r!}\sum_{i\ge 0}i!\sdiff{n}{i}{r} \\
&=|\Cay_{R}[n]|
&& \text{(by Corollary~\ref{cor_endcay_enum})}
\end{align*}
where in the penultimate equality we used that the term corresponding to $r=n$ in
the summation equals
$$
\frac{|R[n]|}{n!}\sum_{i\ge 0}i!\sdiff{n}{i}{n}
= \frac{n!|R[n]|}{n!}\sdiff{n}{n}{n} = |R[n]|.
$$
This concludes the proof.
\end{proof}

An arguably simpler species equation for $\Cay_R$ is
provided in the next result.

\begin{corollary}
We have
$$
\Cay_R
=\frac{1}{2}\sum_{r\ge 0} R_r\odot\left(1+\Set^r\Bal^r\right).
$$
\end{corollary}
\begin{proof}
The $\LL$-species identity
\begin{equation}
\int\Set^r\Bal^{r+1}=\frac{1}{2r}(\Set^r\Bal^r-1).\label{eq_intBal}
\end{equation}
is easily verified using differentiation and the species version of the
fundamental theorem of calculus.
Next,
\begin{align*}
\Cay_{R} &= R + \sum_{r\ge 1}
(R_r)^{\bullet}\odot\left(\int\Set^r\Bal^{r+1}\right)
&& \text{(by Theorem~\ref{thm_uniPsiR})}\\
&= \sum_{r\ge 0}
R_r+(R_r)^{\bullet}\odot\left(\int\Set^r\Bal^{r+1}\right)
&& \text{(since $R=\sum_{r\ge 0}R_r$ \& $R^{\bullet}_0=\emptyset$)}\\
&= \sum_{r\ge 0}
R_r+rR_r\odot\frac{1}{2r}(\Set^r\Bal^r-1)
&& \text{(by equation~\eqref{eq_intBal})}\\
&= \frac{1}{2}\sum_{r\ge 0} R_r\odot\left(1+\Set^r\Bal^r\right),
\end{align*}
where the last identity follows by $R_r=R_r\odot 1$ and linearity of $\odot$.
\end{proof}

Finding a combinatorial proof of this corollary remains an open problem.

\section{Statistics and asymptotics}\label{section_stats}

Theorem~\ref{thm_PsiRr} provides a closed formula
for the number of $\Psi_R$-structures with~$i$ internal nodes,
among which~$r$ are recurrent, and~$j$ leaves. By suitably summing
over~$r$, $i$ and~$j$, we obtain the distribution of each of these
statistics over $R$-recurrent functional digraphs. For instance,
the number of $R$-recurrent functional digraphs with~$n$ nodes
and~$r$ recurrent points is
\[
\sum_{i+j=n}|\Psi_{R_r}[i,j]|
=\frac{|R[r]|}{r!}\sum_{i=0}^ni!\sdiff{n}{i}{r}.
\]
The distributions of internal nodes and leaves are obtained similarly.
The total number of recurrent points, internal nodes and leaves
can be computed in the same fashion
by multiplying the term inside the summation by~$r$, $i$ or~$j$.
An example is the total number of recurrent points
over $\Psi_R$-structures on $[i,j]$:
\[
\sum_{r=0}^i|\Psi_{R_r}[i,j]|\cdot r
=\sum_{r=0}^i\frac{|R[r]|}{r!}i!\sdiff{i+j}{i}{r}\cdot r.
\]
Multiplying by~$r$ has the same effect as pointing to the
$R$-structure. This gives us an elegant species expression for
the total number of recurrent points:
\begin{equation}
\Psi_{R^{\bullet}}(X,Y),\;
\text{ where }
|\Psi_{R^{\bullet}}[i,j]|
=\sum_{r=0}^i|\Psi_{R_r}[i,j]|\cdot r.
\end{equation}

If two species $F$ and $G$ are related by $F=\Set(G)$, with $G(0)=0$,
then~$G$ is said to be the species of \emph{connected} $F$-structures.
In this sense, an $F$-structure is a set of $G$-structures each of
which is called a \emph{connected component}.
We are by now familiar with the species of connected permutations,
that is, cycles $\Cyc$ where $\Sym=\Set(\Cyc)$.
Not every species admits a notion of connected
structure in the realm of $\BB$-species, but it is always
possible~\cite[Proposition 18]{BLL}, under the assumption $F(0)=1$, to construct a virtual
species~$G$ such that $F=\Set(G)$. The species~$G$ is uniquely
determined and is called the \emph{combinatorial logarithm} of~$F$,
denoted by $G=\log F$.

For any unisort species $R$ such that $R(0)=1$, we define the two-sort species of
\emph{connected $R$-recurrent functional digraphs} by letting
the recurrent part be equal to the logarithm of~$R$:
$$
\Psi_{\log R}(X,Y).
$$
Pointing to the $\Set$-structure in $\Set(\log R)=R$
has the effect of multiplying by the number of connected
components. Therefore, the total number of connected components over
$\Psi_R$-structures is obtained as $\Psi_{\Set^{\bullet}(\log R)}$.
Note that $\Set^{\bullet}(X)=X\cdot\Set(X)$, from which
\begin{equation}\label{eq_totcomp}
\Set^{\bullet}(\log R)=R\cdot\log R
\quad\text{and}\quad
\Psi_{\Set^{\bullet}(\log R)} = \Psi_{R\log R}
\end{equation}
follows.
Finally, Theorem~\ref{thm_PsiRr} gives the total number of
connected components over $\Psi_R$-structures on $[i,j]$:
\begin{equation}\label{eq_totcompnum}
\sum_{r=0}^i|\Psi_{R\log R}[i,j]|
=\sum_{r=0}^i\frac{|(R\log R)[r]|}{r!}i!\sdiff{i+j}{i}{r}.
\end{equation}

To illustrate the techniques introduced in this section, we will show
that permutations, Cayley permutations and endofunctions share the
following property:

\emph{The total number of recurrent points over connected structures of
  a given size equals the total number of structures of that size.}


First, we consider permutations. There are $(n-1)!$ connected
permutations (cycles) of size~$n\geq 1$. Since every point in a
permutation is recurrent, the total number of recurrent points over
connected structures is $(n-1)!\cdot n=n!=|\Sym[n]|$.  To deal with
Cayley permutations and endofunctions we use the species of connected
functional digraphs, $\Psi_{\Cyc}$. Since $\Cyc'=\Lin$, we find that
$\Cyc^{\bullet} = X\cdot \Cyc' = X\cdot \Lin = \Lin_+$.
Thus, the total number of recurrent points over $\Psi_{\Cyc}$ is given by
\begin{align*}
|\Psi_{\Cyc^{\bullet}}[i,j]|
&=|\Psi_{\Lin_+}[i,j]|
=|\Psi_{\Sym_+}[i,j]|.
\end{align*}
Finally, we translate the previous identity to functional digraphs of
Cayley permutations and endofunctions to establish the desired property:
\begin{align*}
|\End_{\Cyc^{\bullet}}[n]|
&=|\End_{\Sym}[n]|=|\End[n]|;\\
|\Cay_{\Cyc^{\bullet}}[n]|
&=|\Cay_{\Sym}[n]|=|\Cay[n]|.
\end{align*}
The crucial property that makes the above derivation work is
$\Cyc^{\bullet} \equiv \Sym_+$. In other words, the species $F=\Sym_+$
satisfies $(\log F)^{\bullet}\equiv F_+$. The
solutions to this differential equation are precisely the \emph{$k$-colored
  permutations}. That is, $F \equiv \Sym(kX)$ for some positive integer
$k$.

Next, we consider the total number of cycles/connected components
over structures in $\Psi_{\Sym}[i,j]$. It is simply obtained by
letting $R=\Sym$ in equation~\eqref{eq_totcomp}, where
$\Set^{\bullet}(\log\Sym)=\Set^{\bullet}(\Cyc)$ gives
the total number of cycles over all permutations.
It is well known
that the average number
of cycles over permutations of $[n]$ is the $n$th harmonic number
$\Har{n}=\sum_{k=1}^n 1/k$. To see this, note that
$\Set^{\bullet}(\Cyc)=\Sym\cdot\log(\Sym)=\Sym\cdot\Cyc$
and
\[
\Set^{\bullet}(\Cyc)(x)
=\bigl(\Sym\cdot\Cyc\bigr)(x)
=\frac{1}{1-x}\cdot\log\left(\frac{1}{1-x}\right).
\]
Therefore,
$$
\bigl|\Set^{\bullet}(\Cyc)[n]\bigr|
=\bigl|(\Sym\cdot\Cyc)[n]\bigr|
=\sum_{k=1}^n\binom{n}{k}(n-k)!(k-1)!=n!\Har{n}.
$$
Now, by identity~\eqref{eq_totcompnum}, the total
number of cycles over $\Psi_{\Sym}$-structures on $[i,j]$ is
$$
|\Psi_{\Sym\log\Sym}[i,j]|
=i!\sum_{r=0}^i\frac{r!\Har{r}}{r!}\sdiff{i+j}{i}{r}
=i!\sum_{r=0}^i\Har{r}\sdiff{i+j}{i}{r}.
$$
Corollary~\ref{cor_endcay_enum} yields the total
number of cycles over endofunctions (see A190314~\cite{Sl}) and
Cayley permutations of size~$n$:
\begin{align*}
|\End_{\Set^{\bullet}(\Cyc)}[n]|
&=\sum_{i=0}^n\ff{n}{i}\sum_{r=0}^i\Har{r}\sdiff{n}{i}{r};\\
|\Cay_{\Set^{\bullet}(\Cyc)}[n]|
&=\sum_{i=0}^n\,i!\,\sum_{r=0}^i\Har{r}\sdiff{n}{i}{r}.
\end{align*}
The first few numbers of the latter sequence are
$$
0, 1, 4, 20, 126, 966, 8754, 91686, 1090578, 14528502, 214337874, 3469418646.
$$
At the time of writing, this sequence is not in the OEIS~\cite{Sl}.

The average number of cycles over $\End[n]$ is asymptotic to
$\frac{1}{2}(\log (2n)+\gamma)$, where Euler's constant $\gamma$
is defined as the limiting difference between the harmonic series
and the natural logarithm.
What is the corresponding limit for the average number of cycles
over Cayley permutations? Is there a generalization of Euler's
constant for the average number of cycles over $R$-recurrent
functional digraphs?

\section{Generalizations}\label{section_generalizations}

Equation~\eqref{eqdiff_psiR} in Section~\ref{section_diffeq}
characterizes the two-sort species $\Psi_R$.
There, linear orders determine the shape of the
branches---$\Lin(X)\cdot Y$---that are recursively appended to the
recurrent part to obtain $\Psi_R$-structures.
Replacing~$\Lin$ with any unisort species~$T$ allows us to manipulate
the shape of these branches. Define the species $\Psi_{R,T}(X,Y)$
accordingly as the solution to the system of combinatorial equations
\[
  \ddy{}\Psi_{R,T}(X,Y) = T(X)\cdot\Psi_{R,T}\pointX(X,Y),\qquad \Psi_{R,T}(X,0)=R(X).
\]
For instance, if we let $T=1$ be the species characteristic of the empty set, then we obtain structures where
all the leaves have distance one from a recurrent point (since each
branch $T(X)\cdot Y=Y$ consists of a single leaf).
Letting $R=\Set$, we end up with a ``brush-shaped'' endofunction,
that is, a set of roots (of sort~$X$) to each of which is appended a set
of leaves (of sort~$Y$):

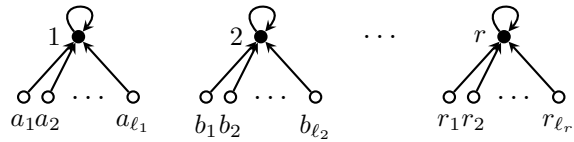
\begin{center}
\begin{tikzpicture}[
thick,
baseline=0pt,
scale=0.8]
\node at (1.2,0) {$\dots$};
\node at (4.2,0) {$\dots$};
\node at (8.2,0) {$\dots$};
\node at (6,1) {$\dots$};
\begin{scope}[font = {\small},dot/.style = {draw,circle,inner sep=1.5pt, node contents={}, label=#1},
every loop/.style={min distance=7mm,in=50,out=130,looseness=10},]
\node[fill=black] (1) at (1,1) [dot=left:$1$];
\path (1) edge[loop, -stealth] (1);
\node (a1) at (0.1,0) [dot=below:{$a_1$}];
\node (b1) at (0.5,0) [dot=below:{$a_2$}];
\node (d1) at (1.9,0) [dot=below:{$a_{\ell_1}$}];
\path (a1) edge[-stealth] (1);
\path (b1) edge[-stealth] (1);
\path (d1) edge[-stealth] (1);
\node[fill=black] (2) at (4,1) [dot=left:$2$];
\path (2) edge[loop, -stealth] (2);
\node (a2) at (3.1,0) [dot=below:{$b_1$}];
\node (b2) at (3.5,0) [dot=below:{$b_2$}];
\node (d2) at (4.9,0) [dot=below:{$b_{\ell_2}$}];
\path (a2) edge[-stealth] (2);
\path (b2) edge[-stealth] (2);
\path (d2) edge[-stealth] (2);
\node[fill=black] (3) at (8,1) [dot=left:$r$];
\path (3) edge[loop, -stealth] (3);
\node (a3) at (7.1,0) [dot=below:{$r_1$}];
\node (b3) at (7.5,0) [dot=below:{$r_2$}];
\node (d3) at (8.9,0) [dot=below:{$r_{\ell_r}$}];
\path (a3) edge[-stealth] (3);
\path (b3) edge[-stealth] (3);
\path (d3) edge[-stealth] (3);
\end{scope}
\end{tikzpicture}
\end{center}

Note that
$$
\Psi_{\Set,1}(X,Y)=\Set\bigl(X\cdot \Set(Y)\bigr).
$$
The corresponding endofunctions of $\End_{\Set,1}$ and Cayley
permutations of $\Cay_{\Set,1}$ are precisely those~$f$ such that
$f^{(2)}=f$, where $f^{(2)}=f\circ f$ denotes the composition
of~$f$ with itself (see also A000248 and A026898~\cite{Sl}).
In other words, $\Psi_{\Set,1}$ is the species of \emph{idempotent
functional digraphs} of two-sorts.

Let us push this construction a bit further. Setting $R=\Sym$ has the
effect of allowing cycles of any length, yielding maps~$f$ where
$f^{(k)}=f$ for some $k>1$; in this case,
$$
\Psi_{\Sym,1}(X,Y)=\Sym\bigl(X\cdot \Set(Y)\bigr).
$$
In the same spirit, we fix $k$ and let
$$
R=\Set\circ\sum_{i|k}\Cyc_i
$$
be the species of permutations whose cycle lengths divide~$k$.
Then $\Psi_{R,1}$ yields maps $f$ such that $f^{(k+1)}=f$.
What other choices of $R$ and $T$ lead to interesting examples?

In Section~\ref{section_Rrecdig}, we have defined the two-sort species
$\Psi_R=R\circ\rtree$ whose structures are obtained by arranging
a set of rooted trees in an $R$-structure.
We can tweak this definition by replacing rooted trees with any
other class of trees. For instance, let $\btree=\btree(X,Y)$ be the
two-sort species of rooted trees
where each node has at most two (unordered) children. They are
obtained as
$$
\btree(X,Y) =
X\cdot \Bigl(\bigl( 1 + X + \Set_2\bigr) \circ \bigl(\btree(X,Y) - X + Y\bigr)\bigr).
$$
Then, $\Sym\circ\btree(X,Y)$ yields functional digraphs where each
recurrent point has indegree at most three, and each nonrecurrent
point has indegree at most two.
In the corresponding endofunctions and Cayley permutations, obtained
as usual via~\eqref{def_FXX} and~\eqref{def_HATF}, recurrent points
have at most three preimages, and nonrecurrent points have at most
two preimages (see also A201996~\cite{Sl}).
More generally, rooted trees where each node has at most $k$ children
$$
\mathcal{K}(X,Y) =
X\cdot \Bigl(\bigl( 1 + X + \Set_2 + \cdots + \Set_k\bigr) \circ \bigl(\mathcal{K}(X,Y) - X + Y\bigr)\bigr)
$$
determine maps where recurrent points have at most $k+1$ preimages
and nonrecurrent points have at most~$k$ preimages.
What can we say about other classes of trees and choices of $R$?

\appendix

\section{Proof of Proposition~\ref{prop:cayder-asymp}}\label{appendix_asymp}

Recall that Proposition~\ref{prop:cayder-asymp} refines
Theorem~\ref{main-theorem} by giving the rate of convergence.
We establish five preparatory lemmas before giving the proof.
The key ingredients are a bound on pairwise intersections of the
sets $S(a,b;n,i)$ introduced before Lemma~\ref{lem:merge},
a uniform estimate on~$p_r(n)$ obtained via inclusion--exclusion,
and an explicit series evaluation.

Truncating the inclusion--exclusion formula for a finite union
$\bigcup_{j=1}^{M} A_j$ after the first or second term yields,
respectively, an upper and a lower bound:
\[
  \sum_{j=1}^{M} |A_j| - \sum_{1 \leq j < k \leq M} |A_j \cap A_k|
  \leq \biggl|\,\bigcup_{j=1}^{M} A_j\,\biggr|
  \leq \sum_{j=1}^{M} |A_j|.
\]
These are known as the first two \emph{Bonferroni inequalities}.
To apply them in our setting we need to control pairwise
intersections, which is the purpose of the next lemma.
Recall that $S(a,b;n,i)$ denotes the set of partitions of~$[n]$
into $i$~blocks in which $a$ and~$b$ lie in the same block
(defined before Lemma~\ref{lem:merge}).

\begin{lemma}\label{lem:intersection}
For $n \geq 3$ and distinct $2$-element subsets
$\{a,b\}, \{c,d\} \subseteq [n]$, and for all $i\ge 0$,
\[
  |S(a,b;n,i) \cap S(c,d;n,i)| \leq \setstirling{n-2}{i}.
\]
\end{lemma}

\begin{proof}
A partition in $S(a,b;n,i) \cap S(c,d;n,i)$ places $a,b$ in a
common block and $c,d$ in a common block.
Since $\{a,b\} \neq \{c,d\}$, both
$\{a,b\} \setminus \{c,d\}$ and $\{c,d\} \setminus \{a,b\}$
are nonempty.
Choose $x \in \{a,b\} \setminus \{c,d\}$ and
$y \in \{c,d\} \setminus \{a,b\}$;
note that $x \neq y$ since $x \in \{a,b\}$ while
$y \notin \{a,b\}$.
Write $x'$ for the other element of $\{a,b\}$ and $y'$ for the
other element of $\{c,d\}$, so that $x$ and~$x'$ lie in the same
block and $y$ and~$y'$ lie in the same block.

We construct an injection from $S(a,b;n,i) \cap S(c,d;n,i)$ into
the set of partitions of $[n] \setminus \{x,y\}$ into $i$~blocks
by deleting $x$ and~$y$ from their respective blocks.
Since the block of~$x$ also contains~$x' \neq x$, it remains
nonempty after removing~$x$; likewise for the block of~$y$, which
contains~$y' \neq y$.
Thus no blocks are lost, and the result is a partition of the
$(n-2)$-element set $[n] \setminus \{x,y\}$ into exactly
$i$~blocks.

To see that the map is injective, note that the original partition
can be recovered by reinserting~$x$ into the block of~$x'$ and~$y$
into the block of~$y'$.
This is well-defined because both $x'$ and $y'$ remain in
$[n]\setminus\{x,y\}$: indeed $x' \neq x$ by definition,
and $x' \neq y$ since $x' \in \{a,b\}$ while $y \notin \{a,b\}$;
similarly $y' \neq y$ and $y' \neq x$.
Since the codomain has $\setstirling{n-2}{i}$ elements, the bound
follows.
\end{proof}

With the intersection bound in hand, we can obtain a uniform
estimate on~$p_r(n)$.
We first record the Fubini ratios that will be needed.

\begin{lemma}\label{lem:fubini-ratios}
With $r_1$ as in Lemma~\ref{lem:fubini-asymp},
\begin{align*}
  \frac{\fub{n-1}}{\fub{n}}
  =\frac{\log 2}{n}\Bigl(1+O(r_1^{\,n})\Bigr)
  \quad & \text{and} \quad
  \frac{\fub{n-2}}{\fub{n}}
  =\frac{(\log 2)^2}{n(n-1)}\Bigl(1+O(r_1^{\,n})\Bigr). \\
\intertext{Since $r_1^{\,n} = o(1/n^2)$, these give}
  \frac{\fub{n-1}}{\fub{n}}
  = \frac{\log 2}{n} + O\!\left(\frac{1}{n^2}\right)
  \quad & \text{and} \quad
  \frac{\fub{n-2}}{\fub{n}}
  = O\!\left(\frac{1}{n^2}\right).
\end{align*}
\end{lemma}

\begin{proof}
By Lemma~\ref{lem:fubini-asymp},
$\fub{n} = \frac{1}{2} n!(\log 2)^{-n-1}
\bigl(1+O(r_1^{\,n})\bigr)$, so
\[
  \frac{\fub{n-1}}{\fub{n}}
  = \frac{(n-1)!\,(\log 2)^{n+1}}
         {n!\,(\log 2)^{n}}
    \cdot \frac{1+O(r_1^{\,n})}{1+O(r_1^{\,n})}
  = \frac{\log 2}{n}\Bigl(1+O(r_1^{\,n})\Bigr),
\]
and similarly for $\fub{n-2}/\fub{n}$.
\end{proof}

\begin{lemma}\label{lem:uniform-pr}
As $n\to\infty$, uniformly for $0 \leq r \leq n$,
\[
  p_r(n) = \frac{r\log 2}{n} + O\!\left(\frac{r^4}{n^2}\right).
\]
In other words, there exists a constant $C > 0$ such that for all
sufficiently large~$n$,
$|p_r(n) - r\log 2/n| \leq C\,r^4/n^2$
for every $0 \leq r \leq n$.
\end{lemma}

\begin{proof}
As in the proof of Lemma~\ref{lem:union},
$\setstirling{n}{i} - \stirl{n}{i}{r}
= \bigl|\bigcup_{\{a,b\} \subseteq [r]} S(a,b;n,i)\bigr|$,
a union of $M = \binom{r}{2}$ sets.
By Lemma~\ref{lem:merge},
each has size $\setstirling{n-1}{i}$, and by
Lemma~\ref{lem:intersection}, each pairwise intersection has
size at most $\setstirling{n-2}{i}$.
The Bonferroni inequalities therefore give
\[
  M\,\setstirling{n-1}{i} - \binom{M}{2}\setstirling{n-2}{i}
  \leq \setstirling{n}{i} - \stirl{n}{i}{r}
  \leq M\,\setstirling{n-1}{i}.
\]
In particular, $\setstirling{n}{i} - \stirl{n}{i}{r}$ differs from
$M\,\setstirling{n-1}{i}$ by at most
$\binom{M}{2}\,\setstirling{n-2}{i}$.
Since $\sum_i i!\,\setstirling{m}{i} = \fub{m}$, multiplying
by~$i!$ and summing over~$i$ gives
\begin{align*}
  \fub{n} - \rfub{n}{r}
  &= M\,\fub{n-1}
    + O\bigl(\tbinom{M}{2}\,\fub{n-2}\bigr) \\
  &= \binom{r}{2}\,\fub{n-1}
    + O\bigl(r^4\,\fub{n-2}\bigr),
\end{align*}
where the second equality substitutes $M = \binom{r}{2}$
and $\binom{M}{2} \leq r^4/8$.

To pass from $\rfub{n}{r}$ to $\sfub{n}{r}$, we compute
\begin{align*}
  \sfub{n}{r}
  &= \rfub{n}{r} - \rfub{n}{r+1} \\
  &= (\fub{n} - \rfub{n}{r+1})
     - (\fub{n} - \rfub{n}{r}) \\
  &= \binom{r+1}{2}\fub{n-1} - \binom{r}{2}\fub{n-1}
     + O(r^4\,\fub{n-2}) \\
  &= r\,\fub{n-1} + O\bigl(r^4\,\fub{n-2}\bigr).
\end{align*}
Dividing by $\fub{n}$ gives
\[
  p_r(n) = r\,\frac{\fub{n-1}}{\fub{n}}
  + O\!\left(r^4\,\frac{\fub{n-2}}{\fub{n}}\right).
\]
By Lemma~\ref{lem:fubini-ratios}, the second term is
$O(r^4/n^2)$, while the first is
$r\log 2/n + O(r/n^2)$.
Since $r \leq r^4$ for $r \geq 1$ and both sides vanish
for $r = 0$, the result follows.
\end{proof}

It remains to evaluate the series $\sum r\,\epsilon_r$ that will
arise when we substitute the estimate of
Lemma~\ref{lem:uniform-pr}.

\begin{lemma}\label{lem:series-eval}
$\displaystyle\sum_{r \geq 0} r\,\epsilon_r = -\frac{1}{2e}$.
\end{lemma}

\begin{proof}
By Lemma~\ref{lem:error},
$\epsilon_r = -\sum_{k \geq r+1} (-1)^k/k!$, so
\[
  \sum_{r \geq 0} r\,\epsilon_r
  = -\sum_{r \geq 0}\, \sum_{k \geq r+1}
    \frac{r\,(-1)^k}{k!}.
\]
Since $\sum_{r \geq 0} r\,|\epsilon_r| < \infty$, we may
interchange the order of summation.
The sum over~$r$ then runs from~$0$ to~$k-1$, contributing
$\sum_{r=0}^{k-1} r = \binom{k}{2}$, which vanishes for $k \leq 1$.
Thus
\[
  \sum_{r \geq 0} r\,\epsilon_r
  = -\sum_{k \geq 2} \frac{(-1)^k}{k!} \binom{k}{2}
  = -\frac{1}{2}\sum_{k \geq 2} \frac{(-1)^k}{(k-2)!}
  = -\frac{1}{2}\sum_{j \geq 0} \frac{(-1)^j}{j!} = -\frac{1}{2e}.\qedhere
\]
\end{proof}

Finally, we verify that the splitting point $K = \lfloor \log n \rfloor$
grows fast enough for the tail to be negligible.

\begin{lemma}\label{lem:factorial-tail}
For $K = \lfloor \log n \rfloor$, we have
$1/(K+1)! = o(1/n)$.
\end{lemma}

\begin{proof}
Set $m = K + 1 = \lfloor \log n \rfloor + 1$, so
$m \sim \log n$.  Stirling's formula gives
$\log(m!) = m\log m - m + O(\log m)$.
Since $m\log m \sim \log n\cdot\log\log n$, we have
$\log(m!) - \log n \to \infty$ and hence $m!/n \to \infty$.
\end{proof}

\begin{proof}[Proof of Proposition~\ref{prop:cayder-asymp}]
By Lemma~\ref{lem:error}, it suffices to show that
$\sum_{r=0}^{n} \epsilon_r\, p_r(n) \sim -\log 2/(2en)$.
We split at $K = \lfloor\log n\rfloor$:
\[
  \sum_{r=0}^{n} \epsilon_r\, p_r(n)
  = \sum_{r=0}^{K-1} \epsilon_r\, p_r(n)
  + \sum_{r=K}^{n} \epsilon_r\, p_r(n).
\]
For the tail, $|\epsilon_r| \leq 1/(r+1)! \leq 1/(K+1)!$ for all
$r \geq K$, and $\sum_{r=0}^{n} p_r(n) = 1$, so
$\bigl|\sum_{r=K}^{n} \epsilon_r\, p_r(n)\bigr|
\leq 1/(K+1)! = o(1/n)$ by
Lemma~\ref{lem:factorial-tail}.

For the main sum, Lemma~\ref{lem:uniform-pr} gives
$p_r(n) = r\log 2/n + O(r^4/n^2)$ for each $0 \leq r \leq K$.
Substituting,
\begin{equation}\label{eq:main-sum}
  \sum_{r=0}^{K-1} \epsilon_r\, p_r(n)
  = \frac{\log 2}{n}\sum_{r=0}^{K-1} r\,\epsilon_r
    + O\!\left(\frac{1}{n^2}\right),
\end{equation}
where the error uses
$\sum_{r \geq 0} r^4\,|\epsilon_r|
\leq \sum_{r \geq 0} r^4/(r+1)! < \infty$.
Since also $\sum_{r \geq 0} r\,|\epsilon_r| < \infty$,
the partial sum $\sum_{r=0}^{K-1} r\,\epsilon_r$ can be completed
to the full series $\sum_{r \geq 0} r\,\epsilon_r$ at the cost of
an error $\sum_{r \geq K} r\,|\epsilon_r| = O(1/K!) = o(1)$.
Applying Lemma~\ref{lem:series-eval} and collecting errors,
\[
  \sum_{r=0}^{n} \epsilon_r\, p_r(n)
  = -\frac{\log 2}{2en} + o\!\left(\frac{1}{n}\right).
\]
Indeed, the three error terms are all $o(1/n)$:
the tail is $o(1/n)$ by Lemma~\ref{lem:factorial-tail},
the error in~\eqref{eq:main-sum} is $O(1/n^2) = o(1/n)$,
and completing the series costs
$O(1/(nK!)) = o(1/n)$ since $1/K! = o(1)$.
This establishes Proposition~\ref{prop:cayder-asymp}.
\end{proof}

\end{document}